\documentclass[twoside, reqno, 10pt]{amsart}

\usepackage[left=3cm, right=3cm, top=3cm, bottom=3cm]{geometry}

\usepackage[colorlinks=true, pdfstartview=FitV, linkcolor=blue,citecolor=blue, urlcolor=blue]{hyperref}

\usepackage[usenames]{xcolor}
\definecolor{labelkey}{rgb}{0,0,1}
\definecolor{Red}{rgb}{0.7,0,0.1}
\definecolor{Green}{rgb}{0,0.7,0}

\usepackage{amsfonts, amssymb, amsmath, amsthm, mathrsfs, bbm, cjhebrew, gensymb, textcomp, mathtools, dsfont,tikz}

\usepackage[normalem]{ulem}

\usepackage{commath, setspace, subcaption, parcolumns, multirow, multicol, accents, comment, marginnote, verbatim, empheq, enumerate, stackrel, enumitem, stmaryrd }

\usepackage{lineno}

\makeatletter
\def\namedlabel#1#2{\begingroup
    #2%
    \def\@currentlabel{#2}%
    \phantomsection\label{#1}\endgroup
}
\makeatother

\usepackage[capitalize,nameinlink,noabbrev]{cleveref}
\usepackage{graphicx}
\usepackage{epsfig}
\usepackage{psfrag}
\usepackage{float}
\usepackage[active]{srcltx}

\newtheorem{Thm}{Theorem}[section]

\newtheorem{Prop}[Thm]{Proposition}
\newtheorem{Lem}[Thm]{Lemma}
\newtheorem{Rmk}{Remark}[section]

\numberwithin{equation}{section}

\newcommand{\BL}{B_\veps(t)}

\newcommand{\de}{\delta} 
\newcommand{\De}{\Delta}
\newcommand{\eps}{\epsilon}
\newcommand{\veps}{\varepsilon}

\newcommand{\ph}{\varphi}

\newcommand{\lam}{\lambda}

\newcommand{\tht}{\theta}

\newcommand{\om}{\omega}
\newcommand{\Om}{\Omega}
\newcommand{\ze}{\zeta}

\newcommand{\Gr}{\mathscr{G}}

\newcommand{\PP}{\mathbb{P}}
\def \AA {\mathbb{A}}
\def \tN {\mathrm{N}}
\newcommand{\BB}{\mathbb{B}}

\newcommand{\TT}{\mathbb{T}}

\newcommand{\QQ}{\mathbb{Q}}

\newcommand{\lb}{\big\langle}
\newcommand{\rb}{\big\rangle}
\newcommand{\rrb}{\rrbracket}
\newcommand{\llb}{\llbracket}
\newcommand{\lbn}{\left\langle}
\newcommand{\rbn}{\right\rangle}
\newcommand{\goesto}{\rightarrow}

\newcommand{\Sob}[2]{\lVert#1\rVert_{#2}}
\newcommand{\nrm}[1]{\lVert#1\rVert}

\newcommand{\bdy}{\partial}

\newcommand{\til}[1]{\widetilde{#1}}

\newcommand{\Acal}{\mathcal{A}}

\newcommand{\Tr}{\mathrm{Tr}}

\DeclareMathOperator{\vspan}{span}
\DeclareMathOperator{\dist}{dist}

\usepackage{tikz}

\title[Dimension of the global attractor of 2D NSE on
$\beta$-plane]{Upper bounds on the dimension of the global attractor of the 2D {N}avier-{S}tokes equations on the $\beta-$plane}
\author{Aseel Farhat, Anuj Kumar, Vincent R. Martinez}

\begin{document}

\begin{abstract} 
This article establishes estimates on the dimension of the global attractor of the two-dimensional rotating Navier-Stokes equation for viscous, incompressible fluids on the $\beta$-plane. Previous results in this setting  by M.A.H. Al-Jaboori and D. Wirosoetisno (2011) had proved that the global attractor collapses to a single point that depends only the {latitudinal} coordinate, i.e., \textit{zonal flow}, when the rotation is sufficiently fast. However, an explicit quantification of the complexity of the global attractor in terms of $\beta$ had remained open. In this paper, such estimates are established which are valid across a wide regime of rotation rates and are consistent with the dynamically degenerate regime previously identified. Additionally, a decomposition of solutions is established detailing the asymptotic behavior of the solutions in the limit of large rotation.
\end{abstract}

\maketitle

{
}

\section{Introduction}
In this paper, we consider the evolution of a viscous, incompressible fluid in the presence of a Coriolis force and study properties of its long-time behavior.  In geophysics, the Coriolis effect is commonly approximated to first order by a linear expression of the form $f_0+\beta x_2$ where, $x_2$ represents latitude and $\beta$ {represents the local gradient of the Coriolis parameter.}
This is known as the $\beta$-plane approximation and it accounts for the geophysical fact that for a fluid evolving on a rotating spherical body, the subsequent Coriolis force that acts on the fluid is latitude-dependent. In the study of such flows, a typical simplification for describing surface waves is to assume constant vertical motion \cite{Ped87}. If one additionally models the effects of surface stress and friction, one may arrive at the externally driven, two-dimensional (2D) Navier-Stokes equations (NSE) for an incompressible fluid in the $\beta-$plane approximation. Over a domain ${\Omega}=[0,L]\times[-L/2,L/2]$ and equipped with periodic boundary conditions, these equations are given by
    \begin{align}\label{eq:NSE:rotating:d}
        \partial_tu+(u\cdot \nabla) u=\nu\Delta u -\beta x_2 u^{\perp}-\nabla p+F,\quad \nabla\cdotp u=0,
    \end{align}
where $\nu$ denotes the kinematic viscosity, $u=(u^1,u^2)$ the velocity vector field of the fluid, $u^\perp=(-u^2,u^1)$ its rotation $90^\circ$ counter-clockwise, $p$ the scalar pressure field, and $F=(F^1,F^2)$ is a given, time-independent external driving force. Upon appropriately re-scaling \eqref{eq:NSE:rotating:d}, {one obtains the following} non-dimensionalized form {of the system}:
    \begin{align}\label{eq:NSE:rotating}
        \partial_{t'} u'+(u'\cdot \nabla') u'=\Delta' u' - \frac{1}{\varepsilon} x'_2 (u')^{\perp}-\nabla' p'+\Gr F',\quad \nabla'\cdot u'=0,
    \end{align}
posed over the re-scaled domain $\Om'=[0,2\pi]\times[-\pi,\pi]$, where each of the quantities $t',x', u', p', F'$ and operators $\partial_{t'}, \nabla', \De'$ are now dimensionless. The dimensionless parameters $\veps$ and $\Gr$ are referred to as the \textit{Rossby number} and \textit{Grashof number}, respectively, where the Grashof number is defined by:
    \begin{align}\label{def:Gr}
        \Gr=\frac{\ell^2\Sob{F}{L^2}}{\nu^2},
    \end{align}
for some fixed length scale $\ell$. For the remainder of the manuscript, we will drop the notation $'$. Additionally, we will denote by $A\lesssim B$ to mean that $A\le c B$ holds for some generic positive non-dimensional constant $c$ which is independent of $\Gr$ and $\varepsilon$.

One of the main goals of the present study is to obtain an upper bound on the Hausdorff dimension of the global attractor, $\Acal^\veps$, corresponding to \eqref{eq:NSE:rotating}, that explicitly exhibits dependence on the Rossby number. We point out that due to the skew-symmetry of $u^\perp$, the presence of the Coriolis force plays no apparent role in the approach of \cite{ConstantinFoiasTemam1988} since it vanishes identically in obtaining lower bounds on the trace of the linearized evolution. Similarly, the same apriori bounds {in $L^2$-based Sobolev spaces} obtained for the non-rotating case are as well available for \eqref{eq:NSE:rotating} (see, for instance, \cite{MustafaDjoko2011}). Thus,
the existence of a finite-dimensional global attractor for \eqref{eq:NSE:rotating} follows exactly as in the non-rotating case. Indeed, in \cite{ConstantinFoiasTemam1988}, P. Constantin, C. Foias, and R. Temam obtain the following estimate, originally in the context of the non-rotating case:
    \begin{align}\label{est:CFT}
        \text{dim}_H \Acal^\veps\lesssim \Gr^{{2}/{3}}(1+\log \Gr)^{{1}/{3}}.
    \end{align}
This is the best unconditional estimate on the dimension of the global attractor and it is known to be sharp \cite{Liu1993, IlyinMiranvilleTiti2004}. Thus, quantifying the effect of the Coriolis force on the size of the dimension of the global attractor requires one to exploit the skew-symmetry in a different way.

{It follows from \eqref{eq:NSE:rotating:d} that the Rossby number only directly affects the non-zonal (in 2D) or baroclinic (in 3D) component of the dynamics. It is well-known, however, that the effect on those components becomes more prominent in the regime of small Rossby number, where the solution has a tendency to become less non-zonal (in 2D), that is, predominantly zonal (a function only of the latitude $x_2$) \cite{MustafaDjoko2011, Djoko2019}, or less baroclinic (in 3D), that is, predominantly barotropic (a function only of the horizontal variables, $x_1, x_2$) \cite{GallagherSaint-Raymond2006}}.  This mechanism was given mathematical clarity by A. Babin, A. Mahalov, B. Nicolaenko in their celebrated work \cite{BabinMahalovNicoleanko_Euler_NSE_rotating}, wherein it was found that the regime of fast rotation offered a stabilizing mechanism that one could exploit in order to effectively reduce the dimensionality of the three-dimensional (3D) Euler and Navier-Stokes systems to either extend the life-span of solutions (in the case of the 3D Euler equations) or establish global-in-time regularity of smooth solutions (in the case of the 3D NSE) (see also the work by Schochet \cite{Schochet1988}). The Coriolis force also introduces dispersive effects into the system. This point of view was developed in, {for instance}, \cite{BabinIlyinTiti2011, KostiankoTitiZelik2018} in the context of the 1D KdV equation, 1D complex Ginzburg-Landau equation, and the 1D Kuramoto-Sivashinsky equation. The stabilization effects arising from the dispersive nature of the rotation were also studied in \cite{ElgindiWidmayer2017} in the context of the 2D Euler equations on the $\beta$-plane. We refer the reader to \cite{CheminDesjardinsGallagherGrenier2006} for a comprehensive study of the dispersive effects of rotation in geophysics.

This mechanism was subsequently exploited by M.A.H. Al-Jaboori and D. Wirosoetisno in \cite{MustafaDjoko2011} in the context of \eqref{eq:NSE:rotating} to show that in the regime of small Rossby number, $\Acal^\veps$ collapses to a single point. Specifically, it directly follows from the proof in \cite{MustafaDjoko2011} that the smallness condition on the Rossby number is dictated by the condition $C\Sob{\widetilde{\om}}{L^2}\lesssim 1$ (see Section 4, \cite{MustafaDjoko2011}). Using the bounds we prove in \cref{lem:om:timeavg} below, we are able to quantify the smallness condition of M.A.H. Al-Jaboori and D. Wirosoetisno and express their result as
\begin{Prop}[\cite{MustafaDjoko2011}]\label{prop:onepoint} Suppose $\Gr\gtrsim1$. If 
    \begin{align}\label{cond:regime:MD}
        \veps \lesssim \Gr^{-{9}/{2}}(1+\log \Gr)^{-1/2},
    \end{align}
then 
    \[
        \dim_H \Acal^\veps=0,
    \]
and the global attractor consists of a single point.
\end{Prop}
This result had confirmed the dimensionality-reduction of the long-time dynamics of \eqref{eq:NSE:rotating} in the small Rossby number regime and indicated that the Hausdorff dimension of $\Acal^\varepsilon$ should in general depend on $\varepsilon$. However, the precise manner in which this degeneration would occur remained unclear.

\par

In this paper, we successfully obtain an upper-bound estimate on the Hausdorff dimension of $\Acal^\varepsilon$ that depends explicitly on $\veps$ in a wider regime of Rossby numbers that is consistent with the dynamically degenerate regime identified in \cref{prop:onepoint}, as well as the unconditional estimate \eqref{est:CFT} established in \cite{ConstantinFoiasTemam1988}. In particular, our main result can be stated as follows:

\begin{Thm}\label{thm:main}

Suppose $\Gr\gtrsim1$. If
    \begin{align}\label{cond:regime:1}
      \veps \lesssim \Gr^{-57/22}(1+\log\Gr)^{-9/22},
    \end{align}
then
    \begin{align}\label{est:regime:1}
        \dim_H \Acal^\veps\lesssim \varepsilon^{1/3} \Gr^{3/2}(1+\log\Gr)^{1/2},
    \end{align}
and if
    \begin{align}\label{cond:regime:2}
       \Gr^{-57/22}(1+\log\Gr)^{-9/22}\lesssim\veps\lesssim\Gr^{-5/2},
    \end{align}
then
    \begin{align}\label{est:regime:2}
        \dim_H \Acal^\veps\lesssim  \varepsilon^4 \Gr^{11}
        (1+\log \Gr)^{2}.
    \end{align}
\end{Thm}

{A few remarks are in order regarding the estimates \eqref{est:regime:1}, \eqref{est:regime:2}. Firstly, although \eqref{est:regime:1} is valid for all $\veps$ sufficiently small, as soon as $\veps\lesssim \Gr^{-{9}/{2}}(1+\log \Gr)^{-1/2}$, then one enters the dynamically degenerate regime of \cref{prop:onepoint}. In particular, when $\veps\lesssim \Gr^{-{9}/{2}}(1+\log \Gr)^{-1/2}$, then one has $\dim_H\Acal^\veps=0$. Nevertheless, we observe that $\Gr^{-9/2}<\Gr^{-57/22}$, whenever $\Gr\gtrsim1$. Note that the maximal estimate in \eqref{est:regime:1} is given by}
    \begin{align}\label{est:regime:1:max}
        \dim_H\Acal^\veps&\lesssim \Gr^{7/11}(1+\log\Gr)^{4/11}.
    \end{align}
{On the other hand, when $\veps\sim \Gr^{-{9}/{2}}(1+\log \Gr)^{-1/2}$, then \eqref{est:regime:1} becomes}
    \begin{align}\label{est:regime:1:min}
        \dim_H\Acal^\veps&\lesssim (1+\log\Gr)^{1/3}.
    \end{align}
Since $7/11<2/3$, the estimate \eqref{est:regime:1:max} yields a non-trivial improvement over \eqref{est:CFT} outside of the dynamically degenerate regime of \eqref{cond:regime:MD}. 

{On the other hand, although \eqref{est:regime:2} holds for all $\veps$ in \eqref{cond:regime:2}, when $\veps\gtrsim \Gr^{-31/12}(1+\log\Gr)^{-5/12}$, then the upper bound estimate \eqref{est:regime:2} actually exceeds the general upper bound asserted in \eqref{est:CFT}. 
Nevertheless, in the range $\Gr^{-57/22}(1+\log\Gr)^{-9/22}\lesssim\veps\lesssim\Gr^{-31/12}(1+\log\Gr)^{-5/12}$ the estimate \eqref{est:regime:2} constitutes a genuine improvement on $\dim_H\Acal^\veps$. Although this regime appears to be rather thin, it may suggest that further investigation is required to understand whether the regimes \eqref{est:regime:1}, \eqref{est:regime:2} meaningfully distinguishes between different dynamical behavior of the system.}

To obtain an estimate for the dimension of the attractor, we use the vorticity formulation \eqref{eq:NSE:rotatingv} below. Our framework is based on the work of C. R. Doering and J. D. Gibbon in \cite{DoeringGibbon1991}, where the authors proved the Constantin-Foias-Temam attractor dimension estimate \eqref{est:CFT} for 2D NSE through the vorticity formulation; {we refer the reader to the book \cite{Doering_Gibbon_1995_book}, where further discussion on this classical result can be found.} A direct replication of their approach establishes \eqref {est:CFT} for \eqref{eq:NSE:rotating} as well, but it remains to be clarified how to obtain an estimate depending on $\veps$. In order to do so, we employ a novel approach to analyze the time-variation of the trace, $ \Tr[\AA \PP_\tN]$, by describing time derivatives of the operator $\PP_\tN$, which denotes the projection operator onto the vector space spanned by evolving volume elements. This critical step allows us to successfully apply differentiation-by-parts developed by \cite{BabinIlyinTiti2011} and independently exploited in \cite{MustafaDjoko2011}, thus resulting in an expression explicitly dependent on $\veps$. 

Regarding the assumption $\Gr\gtrsim1$, we point out that in the context of turbulent flow, this is a very natural assumption. Indeed, it is known that if $\Gr$ is sufficiently small, then the global attractor of the non-rotating 2D NSE consists of a single point \cite{FoiasJollyManleyRosa2002}. Due to the skew-symmetric nature of the rotation, the same argument holds for \eqref{eq:NSE:rotating}. In light of \cref{prop:onepoint}, the assumption $\Gr\gtrsim1$ represents a minimal assumption.

\cref{thm:main} nevertheless leaves a few interesting issues that remain to be addressed. One is whether the improved estimates there are sharp. Another is to understand the manner in which the solutions dynamically approach the zero-Rossby number regime. Our second group of results addresses the latter issue, while the former is reserved for future investigation.

\begin{Thm}\label{thm:decomposition}
Let $f=\Gr F$. Given $\veps>0$, let $u^\veps$ denote the unique solution of \eqref{eq:NSE:rotating} corresponding to initial data $u_0$ and let $\om^\veps=\nabla^\perp\cdotp u^\veps$ denote the corresponding vorticity field {(see \eqref{def:vorticity})}. {We write}
    \begin{align}\label{eq:asymptotic:decomposition}
        \om^\veps(t,x,y)=\overline{\om}(t,y)+\overline{\ze}^\veps(t,y)+\widetilde{\om}^\veps(t,x,y),
    \end{align}
where $\widetilde{\om}^\veps=\om^\veps-\int_0^{2\pi}\om^\veps(t,x,y)dx$ denotes the non-zonal vorticity, $\overline{\om}$ denotes the unique solution of the one-dimensional equation,
    \begin{align}\label{eq:heat}
        \bdy_t\overline{\om}-\bdy_y^2\overline{\om}=\overline{f},\quad \overline{\om}(0,y)=\nabla^\perp\cdotp \overline{u}_0,
    \end{align}
where $\overline{f}=\int_0^{2\pi}f(x,y)dx$, denotes the zonal component of $f$. {Then} 
    \begin{align}\label{eq:zonal:error}
        \bdy_t\overline{\ze}^\veps-\bdy_y^2\overline{\ze}^\veps=-\overline{B}(\om^\veps,\om^\veps),\quad \overline{\ze}^\veps(0,y)=\overline{\om}^\veps_0(y)-\overline{\om}_0(y)
    \end{align}
and
    \begin{align}
        \limsup_{t\goesto\infty}\|\overline{\ze}^\veps(t)\|_{L^2}\leq O(\veps).
    \end{align}
In particular, when $\veps$ satisfies \eqref{cond:regime:MD}, then
    \begin{align}
        \limsup_{t\goesto\infty}\|\om^\veps(t)-\overline{\om}_*-\widetilde{\om}^\veps_*\|_{L^2}\leq O(\veps),
    \end{align}
where $\overline{\om}_*$ denotes the unique stationary solution of \eqref{eq:heat}, and $\widetilde{\om}^\veps_*$ denotes the unique stationary solution of \eqref{eq:NSE:rotating}.
\end{Thm}

Lastly, we observe that a direct consequence of \cref{prop:onepoint} and our result \cref{thm:decomposition} is a continuity result between the global attractors of \eqref{eq:NSE:rotating} and \eqref{eq:heat}.

\begin{Thm}\label{thm:continuity}
Let $\Acal^\veps$ denote the global attractor of \eqref{eq:NSE:rotating} and $\overline{\Acal}$ denote the global attractor of \eqref{eq:heat}.
Then
    \begin{align}\notag
        \lim_{\veps\goesto0^+}\dist(
        \Acal^\veps,\overline{\Acal})=0,
    \end{align}
where $\dist$ denotes the Hausdorff set distance induced by the $H^1$-norm. 
\end{Thm}

The proof of \cref{thm:continuity} is provided in \cref{sect:asymptotic}. Although \cref{thm:continuity} is an expected result in our situation due to the collapse of the global attractor of $\eqref{eq:NSE:rotating}_\veps$ to a single point for all sufficiently small values of $\veps$, it is \textit{not} a statement about the continuity of global attractors between $\eqref{eq:NSE:rotating}_\veps$ and the global attractor, $\Acal^0$, of its \textit{limiting equation} $\eqref{eq:NSE:rotating}_0$, obtained as $\veps\goesto0$. In general, this is a delicate issue and this interesting issue is not treated here. We refer the reader to \cite{KostiankoTitiZelik2018} for an interesting case study in this direction. While upper semicontinuity of global attractors can typically be shown under mild assumptions, lower semicontinuity can often only be guaranteed under more restrictive assumptions \cite{HoangOlsonRobinson2015}. Nevertheless, \cref{thm:continuity} indicates that the failure of continuity between the global attractor, $\Acal^\veps$, of $\eqref{eq:NSE:rotating}_\veps$ and, $\Acal^0$, of $\eqref{eq:NSE:rotating}_0$ could be due to the non-degeneracy of $\Acal^0$.
The main value of \cref{thm:decomposition} and \cref{thm:continuity} is that it demonstrates that the long-time behavior of \eqref{eq:NSE:rotating} in the limit of large rotation is characterized entirely by the long-time behavior of a one-dimensional heat equation. Indeed, \cite[Theorem 3.1]{MustafaDjoko2011} shows that $\sup_{t\geq T_m}\Sob{\widetilde{\om}^\veps(t)}{H^m}=O(\sqrt{\veps})$. 



In the following section (\cref{sec:prelim}), we develop some mathematical preliminaries in order to prove our main theorems above. \cref{sec:prelim} is dedicated to developing the proof of \cref{thm:main}, while \cref{sect:asymptotic} is dedicated to proving \cref{thm:decomposition} and \cref{thm:continuity}. Various technical results needed for these tasks are relegated to \cref{sect:app:A} and \cref{sect:app:B}. In particular, time-averages bounds on the gradient of the non-zonal vorticity are proved in \cref{sect:app:A}, which may be useful for other purposes.

\section{Mathematical Preliminaries}\label{sec:prelim}

For mathematical convenience, we will impose the following symmetry conditions on $u$:
    \begin{align}\label{eq:symmetry}
        u^1(t,x_1,-x_2)=u^1(t,x_1,x_2),\quad u^2(t,x_1,-x_2)=-u^2(t,x_1,x_2).
    \end{align}
It is easily verified that if one specifies an initial condition, $u_0$, and imposes the symmetry \eqref{eq:symmetry} on $u_0$ and $F$, then the corresponding solution, $u(t;u_0)$, will also satisfy \eqref{eq:symmetry}, for all $t\ge 0$. We shall henceforth impose \eqref{eq:symmetry}. Owing to the Galilean invariance of \eqref{eq:NSE:rotating:d} and conservation of the spatial mean, we may further assume that
    \[
        \int_{\Omega}u\,dx=0,\quad \text{and}\quad \int_{\Omega} F\,dx=0.
    \]

\subsection{Vorticity formulation}
Recall that the vorticity $\omega$ is a scalar quantity in 2D, given by
    \begin{align}\label{def:vorticity}
        \omega=\nabla^{\perp}\cdot u=\frac{\partial u^2}{\partial x_1}-\frac{\partial u^1}{\partial x_2}.
    \end{align}
To write the evolution equation for $\omega$, we apply $\nabla^{\perp}\cdot$ to \eqref{eq:NSE:rotating} and obtain 
    \begin{align}\label{eq:NSE:rotatingv}
        \frac{\partial \omega}{\partial t}+u\cdot \nabla \omega=f+ \Delta \omega-\frac{1}{\varepsilon}u^{2},
    \end{align}
where 
    \begin{align}\label{def:f}
        f=\Gr (\nabla^{\perp}\cdot F)
    \end{align}
is \textit{time-independent}.

We assume that $\omega$, $f$ are real-valued and mean-free over $\Omega$.
Consider the bilinear form associated with the nonlinear term in \eqref{eq:NSE:rotatingv}, given by 
    \begin{align}\label{def:B}
        B(\tht_1,\tht_2)=-\nabla^{\perp}(-\Delta)^{-1}\tht_1\cdot \nabla \tht_2.
    \end{align}
Applying integration by parts and using the boundary conditions, it follows that
    \begin{align}\label{E:Bdiv}
        ( B(\tht_1, \tht_2), \tht_2 )=0,
    \end{align}
where $(\,, \,)$ denotes the $L^2-$inner product.
We also denote the operator corresponding to the Coriolis force in \eqref{eq:NSE:rotatingv} by
    \begin{align}\label{def:L}
        L_\veps\tht=-\frac{1}{\varepsilon}\partial_1 (-\Delta)^{-1}\tht.
    \end{align}
Note that since $L$ is skew-adjoint, it follows that
    \begin{align}\label{eq:skew}
        ( L_\veps\tht,\tht)=0.
    \end{align}
We rewrite \eqref{eq:NSE:rotatingv} in the usual functional form as
\begin{align}\label{eq:NSE:rotatingvf}
    \frac{\partial \omega}{\partial t}+B(\omega, \omega)=f+\Delta \omega-L_\veps\omega.
\end{align}

We will make use of a particular orthogonal decomposition of a general scalar field $\tht$ into a purely height-dependent component and its complementary component. Define
\begin{align*}
    \overline{\tht}(x_1,x_2)=\sum_{k_1=0}\hat{\tht}_k e^{ik_2 x_2},\quad
    \widetilde{\tht}(x_1,x_2)=\sum_{k_1 \ne 0}\hat{\tht}_k e^{ik\cdot x},
\end{align*}
where the sum runs over $k \in \mathbb{Z}^2\backslash \{0\}$. We will refer to $\overline{\tht}$ as the {\em zonal} mode and $\widetilde{\tht}$ as the {\em non-zonal} mode. By a simple calculation, it follows that for any scalar fields $\tht_1, \tht_2$, we have
    \begin{align}\label{eq:zonal:cancellation}
        B(\overline{\tht}_1, \overline{\tht}_2)=0.
    \end{align}

For our analysis, it will be useful to rewrite \eqref{eq:NSE:rotatingv} in an equivalent form, where the Coriolis force operator \eqref{def:L} is absorbed into the equation through an integrating factor. This is achieved by applying the change of variables 
    \begin{align}\label{def:eta}
        \omega \mapsto e^{tL_\veps}\omega=:\eta.
    \end{align}
We then obtain the following \textit{non-autonomous} evolution equation for $\eta$:
\begin{align}\label{eq:NSE:rotatingvf:eta}
    \frac{\partial \eta}{\partial t}+\BL(\eta, \eta)=g_\veps(t)+ \Delta \eta,
\end{align}
where $g_\veps=e^{tL_\veps}f$ and
    \begin{align}\label{def:Bveps}
B_\veps(t)(\tht_1, \tht_2)=e^{tL_\veps}B(e^{-tL_\veps}\tht_1, e^{-tL_\veps}\tht_2).
    \end{align}
Observe that the operator $e^{tL_\veps}$ is unitary and satisfies:
    \begin{align}\label{E:unitary:skew}
        \|e^{tL_\veps}\|_{L^2}=1,\qquad (e^{tL_\veps}\tht_1, \tht_2 )=\sum_{k}e^{-i\frac{t}{\varepsilon}\lambda_k}\hat{\tht}_1\overline{\hat{\tht}}_2=( \tht_1, e^{-tL_\veps}\tht_2),
    \end{align}
for all $t\geq0$, where 
    \begin{align}\label{def:lambda}
        \lambda_k =\frac{k_1}{|k|^2}.
    \end{align}

As we remarked earlier, uniform long-time estimates for $\omega$ satisfying rotating Navier-Stokes equations \eqref{eq:NSE:rotatingv} were obtained in \cite{MustafaDjoko2011}. We recall the result from \cite{MustafaDjoko2011} which asserts absorbing ball-type bounds for the vorticity $\om$ in all higher-order Sobolev norms. To state the result, we revert back to physical dimensions and make use of the following notation: For $f=f(x,t)$, we denote by
    \begin{align}\label{def:bracket:norm}
        \llb f \rrb:=\sup_{t>0}\nrm{f(t)}_{L^2}.
    \end{align}

We recall that 2D NSE \eqref{eq:NSE:rotatingvf} is globally well-posed in $V$, has a well-defined semigroup over $V$ for each $f\in H$, and possesses a finite-dimensional global attractor contained in $V$ (see, for instance, \cite{ConstantinFoias1988, Robinson2003, Temam_1997, FoiasManleyRosaTemam2001}). The solutions of \eqref{eq:NSE:rotatingvf} also satisfy the following higher-order apriori bounds, which were proved in \cite{MustafaDjoko2011, Djoko2019}.

\begin{Thm}[\cite{MustafaDjoko2011}]\label{thm:djoko:bounds}
Let the initial data $u_0\in L^2(\Omega)$ be given and the body force $f\in H^{m-1}(\Omega)$, for some $m\geq0$. Then there exists a positive time $T_m=T_m(\nrm{u_0}_{L^2},\llb \nabla^{m-1} f \rrb)$ such that 
    \begin{align}\label{est:om:sobolev}
       &\sup_{t\geq T_m}\nrm{{\omega}(t)}_{H^m}^2+\sup_{t\geq T_m} \left ( e^{-\nu t'}\int_t^{t+t'}e^{\nu s}\Sob{\omega}{H^{m+1}}^2\,ds \right)\notag
       \\
       &\le C_m \frac{L^{2m}\llb \nabla ^{m-1}f \rrb^2}{\nu^2}\left(1+\frac{L^2\llb \nabla ^{-1}f \rrb^2}{\nu^4}\right)^{m},
    \end{align}
for some absolute constant $C_m$, independent of $u_0$.  
\end{Thm}
\begin{Rmk}
For the rest of the paper, $C$ will denote a positive non-dimensional constant and is independent of $\tN, \Gr, \varepsilon$ and $\om_0$, where $\tN$ will represent the dimension of the volume elements to be propagated along the linearized evolution and shown to contract to $0$ (see \cref{sect:attractor} below).
\end{Rmk}

\section{Estimation of the Dimension of the Global Attractor}\label{sect:attractor}
Since we are concerned with the evolution of $\om$ on the global attractor, we have \eqref{est:om:sobolev} holds for all time. Then \eqref{est:om:sobolev} along with \eqref{def:Gr} implies
    \begin{align}\label{est:om:sobolev:Gr}
        \sup_{t\in\mathbb{R}}\Sob{\om(t)}{H^m}\le C (\Gr\vee1)^{m+1}.
    \end{align}
Note that upon letting $\Gr'=\Gr\vee1$, we can get rid of the maximum. It will thus be convenient to simply assume $\Gr\geq1$.

To obtain an upper bound for $\text{dim}_H \Acal^{\varepsilon}$, we follow the standard approach which involves looking at the evolution of arbitrary infinitesimal volume elements under the linearized equations on the attractor. Our goal is to estimate the largest possible N such that all N-dimensional volumes in $L^2(\Omega)$ contract to $0$ asymptotically as $t \to \infty$. 
To this end, we linearize \eqref{eq:NSE:rotatingvf} about a solution on the global attractor, that is, $\om(\cdot;\om_0)$, where $\om_0 \in\Acal^{\varepsilon}$, and obtain
    \begin{align}\label{eq:NSE:rotatingvf:lin:x}
        \frac{\partial \delta x}{\partial t}+B(\omega, \delta x)+B(\delta x, \omega)= \Delta \delta x-L_\veps\delta x.
    \end{align}
Upon formally applying the change of variable $\delta z=e^{tL_\veps}\delta x$, we may rewrite \eqref{eq:NSE:rotatingvf:lin:x} as 
    \begin{align}\label{eq:NSE:rotatingvf:lin:z}
        \frac{\partial \delta z}{\partial t}=-\AA(t)\delta z,
    \end{align}
where $\AA$ denotes the operator
    \begin{align}\label{def:A}
        \AA(t)\phi= -\Delta \phi+B_\veps(t)(\eta, \phi)+B_\veps(t)(\phi, \eta),
    \end{align}
where $\eta$ denotes a solution of \eqref{eq:NSE:rotatingvf:eta}. Let $\delta x_1^0, \delta x_2^0, \dots ,\delta x_N^0$ be N vectors in $L^2 (\Omega)$.

For each $i$, let $\delta x_i(t)$ be the solution of \eqref{eq:NSE:rotatingvf:lin:x} with the initial data $\delta x_i^0$ and let $\delta z_i(t)$ be the solution of \eqref{eq:NSE:rotatingvf:lin:z} with the initial data $\delta z_i^0=\de x_i^0$.
By uniqueness, we have
    \[
        \delta z_i(t)=e^{tL_\veps}\delta x_i(t),\quad t>0.
    \]
We also have that the volume of a parallelepiped is invariant under the rotation operator $e^{tL_\veps}$, i.e. 
\[|\delta x_1(t)\wedge \dots \wedge\delta x_\tN(t)|=|\delta z_1(t)\wedge \dots \wedge\delta z_\tN(t)|.\]
Indeed, one has the following general result. 

\begin{Lem}
Let $H$ be a Hilbert space and $U$ be any unitary operator on $H$. Then for any $N>0$
    \[
        |\delta x_1\wedge \dots \wedge\delta x_\tN|=|\delta z_1\wedge \dots \wedge\delta z_\tN|,
    \]
for all $j=1,\dots,\tN$, where $\delta z_j=U\delta x_j$.
\end{Lem}

From these observations, it suffices to consider the evolution of N-dimensional volume elements under the modified linearized flow defined by $\AA(t)$ in \eqref{eq:NSE:rotatingvf:lin:z}. 
Let
    \[
    V_\tN(t)=|\delta z_1(t)\wedge \dots \wedge\delta z_\tN(t)|.
    \]
Let $\PP_\tN(t)$ denote the projection onto the linear subspace of $L^2(\Omega)$ spanned by $\delta z_1(t),\dots, \delta z_\tN(t)$. The volume element $V_\tN(t)$ evolves according to the following equation (see \cite{Robinson2003}).
    \begin{align}\label{E:Vol}
        V_\tN(t)=V_\tN(0)\exp\left(-\int_0^t \Tr[\AA(s)\PP_\tN(s)]\,ds\right),
    \end{align}
 where $\Tr[\cdot]$ denotes the trace of a trace-class  operator on $L^2(\Omega)$.

To obtain an upper bound for $\text{dim}_H \Acal^\varepsilon$, we look for the smallest value of N for which
    \begin{align}\label{eq:main:claim}
        \langle \Tr[\AA\PP_\tN] \rangle>0,\quad \forall \, \omega(t)\in \Acal^\varepsilon,
    \end{align}
where $\langle\,\cdot \,\rangle$ denotes time average defined by
    \[
        \langle\, h\,\rangle=\limsup_{T \to \infty}\frac{1}{T}\int_0^T h(s) \, ds.
    \]

Let $\{\chi_1, \chi_2, \dots \chi_i,\dots \}$ be a fixed (time-independent) orthonormal basis of $L^2(\Omega)$. Let $\{\phi_1(t), \phi_2(t),\\ \dots \phi_\tN(t)\}$ be an orthonormal basis for $\PP_\tN(t)L^2(\Omega)$. Note that $\phi_j(t)$ may be different from $\chi_1,\dots,\chi_\tN$ at each $t$. In particular, we have 
    \begin{align}\label{eq:orthonormal:L2}
       { \|\chi_j\|_{L^2}^2}=\|\phi_j(t)\|_{L^2}^2=\|e^{tL_\veps}\phi_j(t)\|_{L^2}^2=1
    \end{align}
Now, since the trace is independent of the chosen basis, we have
    \begin{align}\label{eq:independence}
        \Tr[\AA(t)\PP_\tN(t)]=\sum_{j=1}^{\infty}(\AA(t)\PP_\tN(t)\chi_j,\chi_j)=\sum_{j=1}^{\tN}(\AA(t)\phi_j(t),\phi_j(t)).
    \end{align}
Let us also recall an important fact regarding $\Tr$: If $AB$ is trace-class in $L^2$ and  $B$ has finite-dimensional range or if $A,B$ are bounded, then $BA$ is trace-class and 
    \begin{align}\label{eq:commute}
        \Tr[AB]=\Tr[BA].
    \end{align}

Next, we decompose the operator $\AA$. 
Let $\BB$ denote the linear operator defined by
    \begin{align}\label{def:Beps}
         \BB(t)\ph=B_\veps(t)(\ph,\eta(t))    .
    \end{align}
Then observe that
    \begin{align}\label{eq:A:decomposition}
        \AA(t)\ph&=-\De\ph+\AA_0(t)\ph+\overline{\BB}(t)\ph+\widetilde{\BB}(t)\ph,
    \end{align}
where
    \begin{align}\label{def:barB:tilB}
        \AA_0(t)\ph=B_\veps(t)(\eta(t),\ph),\quad\overline{\BB}(t)\ph=B_\veps(t)(\ph,\overline{\eta}(t)),\quad \widetilde{\BB}(t)\ph=B_\veps(t)(\ph,\widetilde{\eta}(t)).
    \end{align}

We now have the following result.

\begin{Lem}\label{lem:tr:split}
For all $\tN>0$
    \begin{align}\label{eq:tr:split}
        \Tr[\AA(t)\PP_\tN(t)]=\Tr[-\De\PP_\tN(t)]+\Tr[\overline{\BB}(t)\PP_\tN(t)]+\Tr[\widetilde{\BB}(t)\PP_\tN(t)].
    \end{align}
\end{Lem}

\begin{proof}

Using \eqref{eq:skew} and \eqref{E:Bdiv}, observe that
    \[
        \Tr[\AA_0(t)\PP_\tN(t)]
        =\sum_{j=1}^{\tN}(B_\veps(t)(\eta(t),\phi_j(t)),\phi_j(t))=\sum_{j=1}^{\tN}(B(\omega,e^{-tL_\veps}\phi_j(t)), e^{-tL_\veps}\phi_j (t))=0.
    \]
Then \eqref{eq:tr:split} follows from \eqref{eq:A:decomposition}.
\end{proof}

Thus, from \cref{lem:tr:split} in order to prove \eqref{eq:main:claim}, it suffices to study the remaining two terms
    \begin{align}\notag
        \Tr[\overline{\BB}(t)\PP_\tN(t)],\quad \Tr[\widetilde{\BB}(t)\PP_\tN(t)].
    \end{align}
In order to do so, observe that from \eqref{eq:independence} and the Rayleigh-Ritz principle (see \cite{Robinson2003}), we have
\begin{align}\label{E:A0}
    \sum_{j=1}^{\tN}\|\nabla\phi_j(t)\|_{L^2}^2=\Tr[- \Delta \PP_\tN(t)]\ge C \tN^2,
\end{align}
for all $t\geq0$. In light of \eqref{eq:orthonormal:L2}, \eqref{eq:independence}, it follows that
    \begin{align}\label{eq:trace:phi}
        \tN=\sum_{j=1}^{\tN}\|\phi_j(t)\|_{L^2}^2 = \sum_{j=1}^{\tN}\|e^{-sL_\veps}\phi_j(t)\|_{L^2}^2\leq C\sqrt{\Tr[-\De\PP_{\tN}(t)]}
    \end{align}
We will also make crucial use of the following result from \cite{DoeringGibbon1991}, which was originally proved in \cite{Constantin1987}, but was proved in a convenient form for the vorticity formulation in \cite{DoeringGibbon1991}. 

\begin{Lem}\label{lem:DG1991}
Suppose $\{\psi_j(t)\}_{j=1}^{\tN}\subset L^2$. Let $\PP_\tN$(t) denote projection onto the subspace $\vspan\{\psi_j(t)\}_{j=1}^{\tN}$. If $\{\nabla \psi_j(t)\}_{j=1}^{\tN}$ is orthonormal in $L^2$, then
    \begin{align}\notag
        \left\Vert \sum_{j=1}^{\tN}|\psi_j(t)|^2\right\Vert_{L^{\infty}}\le C\left(1+\log\left(\Tr[- \Delta \PP_\tN(t)]\right)\right).
    \end{align}
\end{Lem}

\subsection{Treating  $\langle\Tr[\widetilde{\BB}\PP_\tN]\rangle$
} 
We derive the following estimate.
\begin{Lem}\label{lem:tracebound1:final}
    \begin{align}\label{E:tracebound1:final}
        \left\langle \Tr[\widetilde{\BB}\PP_\tN]\right\rangle \ge -C\left\langle \Sob{\nabla \widetilde{\omega}}{L^2}^2 \right\rangle^{1/2} \left \langle\Tr[-\Delta \PP_\tN]\right\rangle^{1/4}\left(1+\log\left \langle \Tr[- \Delta \PP_\tN]\right\rangle\right)^{1/2}.
    \end{align}
\end{Lem}

\begin{proof}

Recall the relation between $\om$ and $\eta$ defined by \eqref{def:eta} and the definition of $B_\veps(t)$ in \eqref{def:Bveps}. Then we observe that 
    \begin{align*}
        \Tr[\widetilde{\BB}(s)\PP_\tN(s)]&= \sum_{j=1}^{\tN}\left(\widetilde{\BB}(s)\phi_j(s),\phi_j(s)\right)
        =\sum_{j=1}^{\tN} \left(B_\veps(s)(\phi_j(s),\widetilde{\eta}), \phi_j(s)\right)\\
        &=\sum_{j=1}^{\tN}\left(B(e^{-sL_\veps}\phi_j(s),\widetilde{\omega}),e^{-sL_\veps}\phi_j(s) \right).
    \end{align*}
For each $j=1,\dots, \tN$, let $\psi_j$ be
    \begin{align}\label{def:psij}
        \psi_j(t)=-\nabla^{\perp}(-\Delta)^{-1}e^{-tL_\veps}\phi_j(t).
    \end{align} 
By the Cauchy-Schwarz inequality and \eqref{eq:orthonormal:L2} we have
    \begin{align}\label{E:CS1}
        \sum_{j=1}^{\tN}\left(B(e^{-sL_\veps}\phi_j(s),\widetilde{\eta}),e^{-sL_\veps}\phi_j(s) \right)&\ge -\int \left[\left(\sum_{j=1}^{\tN}|e^{-sL_\veps}\phi_j|^2\right)^{1/2}\left(\sum_{j=1}^{\tN}|\psi_j(s)|^2\right)^{1/2}|\nabla \widetilde{\omega}|\right]\,dxdy\notag\\
        &\ge -\left\Vert \left(\sum_{j=1}^{\tN}|\psi_j(s)|^2\right)^{1/2}\right\Vert_{L^{\infty}}\left(\sum_{j=1}^{\tN}\|e^{-sL_\veps}\phi_j(s)\|_{L^2}^2\right)^{1/2} \Sob{\nabla \widetilde{\omega}}{L^2}.
    \end{align}
Since $\{\nabla \psi_j(t)\}_{j=1}^{\tN}$ is orthonormal in $L^2$, it follows from  \cref{lem:DG1991} that
    \begin{align}\label{E:basis-Linfinity}
        \left\Vert \left(\sum_{j=1}^{\tN}|\psi_j(s)|^2\right)^{1/2}\right\Vert_{L^{\infty}}\le C\left(1+\log\left(\Tr[- \Delta \PP_\tN(s)]\right)\right)^{1/2}.
    \end{align}
Upon returning to \eqref{E:CS1} and applying \eqref{eq:trace:phi} and \eqref{E:basis-Linfinity}, we obtain
    \begin{align*}
        \Tr[\widetilde{\BB}(s)\PP_\tN(s)]\ge -C\Sob{\nabla \widetilde{\omega}}{L^2}\left(\Tr[-\Delta \PP_\tN(s)]\right)^{1/4}\left(1+\log\left(\Tr[- \Delta \PP_\tN(s)]\right)\right)^{1/2}.
    \end{align*}
Taking the long-time average, applying the Cauchy-Schwarz inequality and then invoking Jensen's inequality for the function $g(y)=\sqrt{y}(1+\log(y))$ for $y>1/e$, we obtain
    \begin{align}\notag
        \left\langle \Tr[\widetilde{\BB}\PP_\tN]\right\rangle \ge -C\left\langle \Sob{\nabla \widetilde{\omega}}{L^2}^2 \right\rangle^{1/2} \left \langle\Tr[-\Delta \PP_\tN]\right\rangle^{1/4}\left(1+\log\left \langle \Tr[- \Delta \PP_\tN]\right\rangle\right)^{1/2},
    \end{align}
which is precisely \eqref{E:tracebound1:final}.
\end{proof}

\subsection{Treating 
$\langle\Tr[\overline{\BB}\PP_\tN]\rangle$
}

First, we decompose the trace of $\overline{\BB}$ further. 

\begin{Lem}\label{lem:Bbar:decomposition}
Define the operators
    \begin{align}\label{def:Bbar1:2}
        \overline{\BB}_1(s)h=\frac{i}{2}e^{sL_\veps}(\partial_y \overline{\omega}e^{-sL_\veps}h),\qquad \overline{\BB}_{2}(s)h=\frac{i}{2}e^{sL_\veps}(\partial_s \partial_y \overline{\omega}e^{-sL_\veps}h).
    \end{align}
Then
    \begin{align}\label{eq:Bbar:decomposition}
        \Tr&[\overline{\BB}(s)\PP_\tN(s)]\notag
        \\
        &=-\varepsilon \partial_s\Tr[\PP_\tN(s)\overline{\BB}_1(s)\PP_\tN(s)]+2\varepsilon \Tr[\PP_\tN(s)\overline{\BB}_1(s)(\partial_s\PP_\tN)(s)]+\varepsilon \Tr[\PP_\tN(s)\overline{\BB}_{2}(s)\PP_\tN(s)].
    \end{align}
\end{Lem}

\begin{proof}
Observe that 
    \[
        \Tr[\overline{\BB}\PP_\tN]=\Tr[\overline{\BB}\PP^2_\tN]=\Tr[\PP_\tN\overline{\BB}\PP_\tN].
    \]
By definition, we have
    \begin{align*}
        \Tr[\PP_\tN(s)\overline{\BB}(s)\PP_\tN(s)]&= \sum_{i=1}^{\infty}\left(\PP_\tN(s)\overline{\BB}(s)\PP_\tN(s)\chi_i,\chi_i\right)
        \\
        &=\sum_{i=1}^{\infty} \left(\overline{\BB}(s)\PP_\tN(s)\chi_i,\PP_\tN(s)\chi_i\right)
        \\
        &=\sum_{i=1}^{\infty} \left(B_\veps(t)(\PP_\tN(s)\chi_i,\overline{\eta}(s)), \PP_\tN(s)\chi_i\right)
        \\
        &=\sum_{i=1}^{\infty}\left(B(e^{-sL_\veps}\PP_\tN(s)\chi_i,\overline{\omega}(s)),e^{-sL_\veps}\PP_\tN(s)\chi_i \right).
\end{align*}
Let us define
    \[
         \Gamma_{jkl}:=\left(B(e^{ij\cdot x},e^{ik\cdot x}),e^{il\cdot x}\right)=4\pi^2 \frac{j_1 k_2-j_2 k_1}{|j|^2}\delta_{j+k-l}.
    \]
Then we have (see \cite{MustafaDjoko2011})
    \begin{align*}\Gamma_{jkl}+\Gamma_{kjl}=
        \begin{cases}
            0,\quad &\text{if} \quad \lambda_j+\lambda_k=0,\quad l_1=0,
            \\
            -4\pi^2l_2(\lambda_j+\lambda_k)\delta_{j+k-l}, \quad & \text{if} \quad l_1=0.
        \end{cases}
    \end{align*}
Upon applying the Plancherel's theorem, we have
    \begin{align}\label{E:IBP}
        &\left(B(e^{-sL_\veps}\PP_\tN(s)\chi_i,e^{-sL_\veps}\PP_\tN(s)\chi_i),\overline{\omega}\right)\notag
        \\
        &=\sum_{j,k,l \in \mathbb{Z}^2}\Gamma_{jkl}\mathcal{F}_j(\PP_\tN(s)\chi_i)\mathcal{F}_k(\PP_\tN(s)\chi_i)\overline{\mathcal{F}_l(\overline{\omega})}e^{-i(\lambda_j+\lambda_k)s/\varepsilon}\notag
        \\
        &=\frac{1}{2}\sum_{j,k,l \in \mathbb{Z}^2}\left(\Gamma_{jkl}+\Gamma_{kjl}\right)\mathcal{F}_j(\PP_\tN(s)\chi_i)\mathcal{F}_k(\PP_\tN(s)\chi_i)\overline{\mathcal{F}_l(\overline{\omega})}e^{-i(\lambda_j+\lambda_k)s/\varepsilon}\notag
        \\
        &=2\pi^2\sum_{j,k,l \in \mathbb{Z}^2,\, l_1=0}-l_2(\lambda_j+\lambda_k)\delta_{j+k-l}\mathcal{F}_j(\PP_\tN(s)\chi_i)\mathcal{F}_k(\PP_\tN(s)\chi_i)\overline{\mathcal{F}_l(\overline{\omega})}e^{-i(\lambda_j+\lambda_k)s/\varepsilon}\notag
        \\
        &=-i2\pi^2\sum_{j,k,l \in \mathbb{Z}^2,\, l_1= 0}\varepsilon l_2\delta_{j+k-l}\mathcal{F}_j(\PP_\tN(s)\chi_i)\mathcal{F}_k(\PP_\tN(s)\chi_i)\overline{\mathcal{F}_l(\overline{\omega})}\partial_s e^{-i(\lambda_j+\lambda_k)s/\varepsilon}\notag
        \\
        &=-i2\pi^2\varepsilon\sum_{j,k,l \in \mathbb{Z}^2,\, l_1= 0} \partial_s \left(\delta_{j+k-l}\mathcal{F}_j(\PP_\tN(s)\chi_i)\mathcal{F}_k(\PP_\tN(s)\chi_i)\overline{\mathcal{F}_l(\partial_y \overline{\omega})}e^{-i(\lambda_j+\lambda_k)s/\varepsilon}\right)\notag
        \\
        &\quad+i2\pi^2\varepsilon\sum_{j,k,l \in \mathbb{Z}^2,\, l_1= 0} \delta_{j+k-l}\mathcal{F}_j(\partial_s \PP_\tN(s)\chi_i)\mathcal{F}_k(\PP_\tN(s)\chi_i)\overline{\mathcal{F}_l(\partial_y \overline{\omega})}e^{-i(\lambda_j+\lambda_k)s/\varepsilon}\notag
        \\
        &\quad+i2\pi^2\varepsilon\sum_{j,k,l \in \mathbb{Z}^2,\, l_1= 0} \delta_{j+k-l} \mathcal{F}_j( \PP_\tN(s)\chi_i)\mathcal{F}_k(\partial_s\PP_\tN(s)\chi_i)\overline{\mathcal{F}_l(\partial_y \overline{\omega})}e^{-i(\lambda_j+\lambda_k)s/\varepsilon}\notag
        \\
        &\quad+i2\pi^2\varepsilon\sum_{j,k,l \in \mathbb{Z}^2,\, l_1= 0} \delta_{j+k-l} \mathcal{F}_j( \PP_\tN(s)\chi_i)\mathcal{F}_k(\PP_\tN(s)\chi_i)\overline{\mathcal{F}_l(\partial_s \partial_y\overline{\omega})}e^{-i(\lambda_j+\lambda_k)s/\varepsilon}\notag
        \\
        &=-\frac{i\varepsilon}{2} \partial_s \left((e^{-sL_\veps}\PP_\tN(s)\chi_i)^2,\partial_y \overline{\omega}\right)+i\varepsilon \left((e^{-sL_\veps}\partial_s \PP_\tN(s)\chi_i)^2,\partial_y \overline{\omega}\right),\notag
        \\
        &\quad+ \frac{i\varepsilon}{2} \left((e^{-sL_\veps}\PP_\tN(s)\chi_i)^2,\partial_s \partial_y \overline{\omega}\right)
\end{align}
Finally, upon summing over $i$, we obtain
\begin{align*}
    \Tr[\PP_\tN(s)\overline{\BB}(s)\PP_\tN(s)]&=-\varepsilon \partial_s\Tr[\PP_\tN(s)\overline{\BB}_1(s)\PP_\tN(s)]+2\varepsilon \Tr[\PP_\tN(s)\overline{\BB}_1(s)(\partial_s\PP_\tN)(s)]\\
    &\quad+\varepsilon \Tr[\PP_\tN(s)\overline{\BB}_{2}(s)\PP_\tN(s)],
\end{align*}
which is precisely \eqref{eq:Bbar:decomposition}.

\end{proof}

Due to the uniform-in-time estimates \eqref{est:om:sobolev} satisfied by $\omega$, we have that $\overline{\BB}_1$ is a bounded operator. We may then deduce from \cref{lem:time:avg:bdd:op} (see \cref{sect:app:A}) that
    \[
        \left\langle \partial_s\Tr[\PP_\tN\overline{\BB}_1\PP_\tN]\right\rangle=0.
    \]
By \cref{lem:Bbar:decomposition}, we are left to treat the long-time average of the next two terms in \eqref{eq:Bbar:decomposition}.

\subsubsection{Treating $\langle \Tr[\PP_\tN\overline{\BB}_1\partial_s\PP_\tN]\rangle$}

The main result of this section is the following estimate.

\begin{Lem}\label{lem:tracebound234:final}
\begin{align}\label{E:tracebound234:final}
     &\lb\Tr[\PP_\tN\overline{\BB}_1\partial_s\PP_\tN]\rb\notag
     \\
     &\geq  -C\left(\sup_{s}\Sob{u(s)}{L^\infty}\Sob{\partial_{y} \overline{\omega}(s)}{L^2}^{\frac{3}{4}}\right) \lbn \Sob{\partial_{y} \overline{\omega}}{L^2}^2 \rbn^{\frac{1}{8}}\lbn \Tr[-\Delta \PP_\tN] \rbn^{\frac{7}{8}}\notag
     \\
     &\quad-C\left(\sup_{s}\Sob{\partial_y \overline{\omega}(s)}{L^\infty}\right)\left\langle \Sob{\nabla {\omega}}{L^2}^2 \right\rangle^{1/2} \left \langle\Tr[-\Delta \PP_\tN]\right\rangle^{1/4}\left(1+\log\left \langle\Tr[- \Delta \PP_\tN]\right\rangle\right)^{1/2}\notag
     \\
     &\quad -C\left(\sup_{s}\Sob{\partial_y \overline{\omega}(s)}{L^\infty}\right) \lbn \Tr[-\Delta \PP_\tN] \rbn-C\left (\sup_{s}\Sob{\partial_{yy} \overline{\omega}(s)}{L^2}^{\frac{3}{4}}\right)  \lbn \Sob{\partial_{yy} \overline{\omega}}{L^2}^2\rbn ^{\frac{1}{8}}\lbn \Tr[-\Delta \PP_\tN] \rbn^{\frac{7}{8}}.
\end{align}
\end{Lem}

To prove \cref{lem:tracebound234:final}, we will make crucial use of the following lemma. 

\begin{Lem}\label{lem:dsPN}
Let $\{\AA(s)\}_{s>0}$ be a family of linear operators such that the following non-autonomous Cauchy problem is well-posed in $L^2$:
    \begin{align}\label{eq:linearization}
         \frac{d z}{d s}=-\AA(s) z,\quad z(0)=z^0.
    \end{align}
Given $z_j^0\in L^2$, for $j=1,\dots, \tN$, let $z_j(s;0)$ denote the unique solution of \eqref{eq:linearization} for $s>0$, such that $z_j(0;0)=z_j^0$. Let $\PP_\tN(s)L^2=\vspan\{z_1(s),\dots, z_{\tN}(s)\}$. Then
for any $h\in L^2$
    \begin{align}\label{eq:dsPNh}
        \frac{d\PP_\tN(s)}{ds} \PP_\tN(s)h=-(\mathbb{I}-\PP_\tN(s))\mathbb{A}(s)\PP_\tN(s)h,
    \end{align}
\end{Lem}

\begin{proof}[Proof of \cref{lem:dsPN}]
By hypothesis, note that it suffices to establish \eqref{eq:dsPNh} for $h=z_j$, for all $j=1,\dots, N$. First, observe that
    \[
        \PP_\tN(s)\ z_j(s)= z_j(s),\quad  1\le j\le \tN.
    \]
Then upon differentiating both sides with respect to $s$ and using \eqref{eq:NSE:rotatingvf:lin:z}, we obtain
    \begin{align*}
        &\frac{d\PP_\tN(s)}{ds} z_j(s)+\PP_\tN(s)\frac{dz_j}{ds}(s)=\frac{dz_j}{ds}(s)=-\AA(s) z_j(s).
    \end{align*}
Applying \eqref{eq:linearization} once again, it follows that
    \[
        \frac{d\PP_\tN(s)}{ds}z_j(s)=-(\mathbb{I}-\PP_\tN(s))\AA(s) z_j(s),
    \]
as desired.
\end{proof}
\begin{Rmk}\label{rmk:dsPN:general}
    Using a similar argument as in the proof of \cref{lem:dsPN}, it follows more generally that for any $h\in L^2$,
     \begin{align*}
        \frac{d\PP_\tN(s)}{ds} h=-(\mathbb{I}-\PP_\tN(s))\mathbb{A}(s)\PP_\tN(s)h-\PP_\tN(s)\mathbb{A}^{*}(s)(\mathbb{I}-\PP_\tN(s))h,
    \end{align*}
    where $\mathbb{A}^{*}(s)$ denotes the adjoint operator of $\mathbb{A}(s)$.
    \end{Rmk}

Let us now prove \cref{lem:tracebound234:final}.

\begin{proof}[Proof of \cref{lem:tracebound234:final}]
Denote by
\begin{align}\label{def:QN}
\QQ_{\tN}=\mathbb{I}-\PP_{\tN}.
\end{align}
Using \eqref{eq:commute}, we may then apply \cref{lem:dsPN} and \eqref{eq:A:decomposition} to obtain
\begin{align}\label{E:tracesplit1}
    &\Tr[\PP_\tN(s)\overline{\BB}_1(s)(\partial_s\PP_\tN)(s)]=\Tr[\overline{\BB}_1(s)(\partial_s\PP_\tN)(s)\PP_\tN(s)]\notag
    \\
    &=-\Tr[\overline{\BB}_1(s)\QQ_{\tN}(s)\AA(s)\PP_{\tN}(s)]\notag
    \\
    &=\Tr[\overline{\BB}_1(s)\QQ_{\tN}(s)(-\Delta)\PP_{\tN}(s)]+\Tr[\overline{\BB}_1(s)\QQ_{\tN}(s)\AA_0(s)\PP_{\tN}(s)]+\Tr[\overline{\BB}_1(s)\QQ_{\tN}(s)\BB(s)\PP_{\tN}(s)]\notag
    \\
    &=I+II+III.
\end{align}

By definition of \eqref{def:QN}, $I$ can be decomposed as
    \begin{align}\label{eq:main:term:PNB1dsPN}
        I=\Tr[\overline{\BB}_1(s)(- \Delta)\PP_{\tN}(s)]+\Tr[\overline{\BB}_1(s)\PP_\tN(s)(- \Delta)\PP_{\tN}(s)]=I_a+I_b.
    \end{align}
Integrating by parts, we have
    \begin{align*}
        |I_a|&=\left|\sum_{j=1}^{\tN}\frac i2\left(e^{sL_\veps}\left(\partial_y \overline{\omega}e^{-sL_\veps}(-\Delta)\phi_j(s)\right),\phi_j(s)\right)\right|\notag
        \\
        &\le  \sum_{j=1}^{\tN}\left|\left(\partial_y \overline{\omega}e^{-sL_\veps}\nabla\phi_j(s)\cdot e^{-sL_\veps} \nabla \phi_j(s)\right)\right|+\sum_{j=1}^{\tN}\left|\left(\partial_{yy} \overline{\omega}e^{-sL_\veps}\partial_y \phi_j(s),e^{-sL_\veps}\phi_j(s)\right)\right|.\notag
    \end{align*}
We have
    \[
        \sum_{j=1}^{\tN}\left|\left(\partial_y \overline{\omega}e^{-sL_\veps}\nabla\phi_j(s)\cdot e^{-sL_\veps} \nabla \phi_j(s)\right)\right|\le \Sob{\partial_y \overline{\omega}}{L^\infty}\left(\sum_{j=1}^{\tN} \Sob{\nabla \phi_j(s)}{L^2}^2\right)\le C\Sob{\partial_y \overline{\omega}}{L^\infty}\Tr[- \Delta \PP_\tN(s)].
    \]
By H\"older's inequality and Agmon's inequality in one-dimension, we have
    \begin{align}\label{E:tracebound2}
        \sum_{j=1}^{\tN}\left|\left(\partial_{yy} \overline{\omega}e^{-sL_\veps}\partial_y \phi_j(s),e^{-sL_\veps}\phi_j(s)\right)\right|&\le \Sob{\partial_{yy} \overline{\omega}}{L^2_y}\int \left(\sum_{j=1}^{\tN} \Sob{e^{-sL_\veps}\partial_y\phi_j(s)}{L^2_y}\Sob{e^{-sL_\veps}\phi_j(s)}{L^\infty_y}\right)\,dx\notag\\
        &\le C\Sob{\partial_{yy} \overline{\omega}}{L^2_y}\int \left(\sum_{j=1}^{\tN} \Sob{e^{-sL_\veps}\partial_y\phi_j(s)}{L^2_y}^{\frac{3}{2}}\Sob{e^{-sL_\veps}\phi_j(s)}{L^2_y}^{\frac{1}{2}}\right)\,dx\notag\\
        &\le C\Sob{\partial_{yy} \overline{\omega}}{L^2_y}\left(\sum_{j=1}^{\tN} \Sob{\nabla\phi_j(s)}{L^2}^{\frac{3}{2}}\Sob{\phi_j(s)}{L^2}^{\frac{1}{2}}\right)\notag\\
        &\le C\Sob{\partial_{yy} \overline{\omega}}{L^2_y} \left(\sum_{j=1}^{\tN} \Sob{\nabla\phi_j(s)}{L^2}^2\right)^{\frac{3}{4}}\left(\sum_{j=1}^{\tN} \Sob{\phi_j(s)}{L^2}^2\right)^{\frac{1}{4}}\notag \\
        &\le C\Sob{\partial_{yy} \overline{\omega}}{L^2}{\Tr[- \Delta \PP_\tN(s)]}^{\frac{7}{8}},
    \end{align}
where we have applied \eqref{E:A0} and \eqref{eq:trace:phi} to obtain the final inequality. Thus, we have
    \begin{align}\label{E:tracebound2a}
        |I_a|
        \le C\Sob{\partial_y \overline{\omega}}{L^\infty}\Tr[- \Delta \PP_\tN(s)]+C\Sob{\partial_{yy} \overline{\omega}}{L^2}{\Tr[- \Delta \PP_\tN(s)]}^{\frac{7}{8}}.
    \end{align}
For $I_b$, we note that $\PP_\tN(s)(- \Delta)\PP_{\tN}(s)$ is a positive operator. Applying \cite[{Page 267}]{ConwayBook}, we have
    \begin{align}\label{E:tracebound2b}
        |I_b|
        &\le \Sob{\overline{\BB}_1(s)}{op}\Tr[\PP_\tN(s)(- \Delta)\PP_{\tN}(s)]\le \Sob{\partial_y \overline{\omega}}{L^\infty}\Tr[-\Delta \PP_\tN(s)].
    \end{align}
Returning to \eqref{eq:main:term:PNB1dsPN}, we now see from \eqref{E:tracebound2a} and \eqref{E:tracebound2b}, that upon taking time average, and applying H\"older's inequality, we obtain
    \begin{align}\label{E:tracebound2:final}
        &\left\langle \Tr[\overline{\BB}_1\QQ_{\tN}(- \Delta)\PP_{\tN}]\right\rangle\notag 
        \\
        &\ge -C\left(\sup_{s}\Sob{\partial_y \overline{\omega}(s)}{L^\infty}\right) \lbn \Tr[-\Delta \PP_\tN] \rbn-C\left(\sup_{s}\Sob{\partial_{yy} \overline{\omega}(s)}{L^2}^{\frac{3}{4}}\right) \lbn \Sob{\partial_{yy} \overline{\omega}}{L^2}^{\frac{1}{4}}\Tr[-\Delta \PP_\tN]^{\frac{7}{8}} \rbn \notag\\
        &
        \ge -C\left(\sup_{s}\Sob{\partial_y \overline{\omega}(s)}{L^\infty}\right) \lbn \Tr[-\Delta \PP_\tN] \rbn-C\left (\sup_{s}\Sob{\partial_{yy} \overline{\omega}(s)}{L^2}^{\frac{3}{4}}\right)  \lbn \Sob{\partial_{yy} \overline{\omega}}{L^2}^2\rbn ^{\frac{1}{8}}\lbn \Tr[-\Delta \PP_\tN] \rbn^{\frac{7}{8}}.
    \end{align}
    
Similarly, for term $II$ in \eqref{E:tracesplit1}, we have
    \begin{align}\label{E:tracebound3}
        |II|&=\left|\Tr[\overline{\BB}_1(s)\QQ_{\tN}(s)\AA_0(s)\PP_{\tN}(s)]\right|\notag
        \\
        &=\left|\sum_{j=1}^{\tN}\frac i2\left(e^{sL_{\veps}}\partial_y \overline{\omega}e^{-sL_\veps}\QQ_{\tN}(s)B_\veps(t)(\eta(s),\phi_j(s)),\phi_j(s)\right)\right|\notag
        \\
        &=\left|\sum_{j=1}^{\tN}\frac i2\left(e^{-sL_\veps}\QQ_{\tN}(s)e^{sL_{\veps}}(u(s)\cdot \nabla e^{-sL_\veps}\phi_j(s)),\partial_y \overline{\omega}e^{-sL_\veps}\phi_j(s)\right)\right|\notag
        \\
        &\le \Sob{\partial_{y} \overline{\omega}}{L^2_y}\int \left(\sum_{j=1}^{\tN} \Sob{e^{-sL_\veps}\QQ_{\tN}(s)e^{sL_{\veps}}(u(s)\cdot \nabla e^{-sL_\veps}\phi_j(s))}{L^2_y}\Sob{e^{-sL_\veps}\phi_j(s)}{L^\infty_y}\right)\,dx\notag
        \\
        &\le C\Sob{\partial_{y} \overline{\omega}}{L^2_y}\int \left(\sum_{j=1}^{\tN}\Sob{e^{-sL_\veps}\QQ_{\tN}(s)e^{sL_{\veps}}(u(s)\cdot \nabla e^{-sL_\veps}\phi_j(s))}{L^2_y} \Sob{e^{-sL_\veps}\partial_y\phi_j(s)}{L^2_y}^{\frac{1}{2}}\Sob{e^{-sL_\veps}\phi_j(s)}{L^2_y}^{\frac{1}{2}}\right)\,dx\notag
        \\
        &\le C\Sob{\partial_{y} \overline{\omega}}{L^2} \Sob{u}{L^\infty}\left(\sum_{j=1}^{\tN} \Sob{\nabla\phi_j(s)}{L^2}^2\right)^{\frac{3}{4}}\left(\sum_{j=1}^{\tN} \Sob{\phi_j(s)}{L^2}^2\right)^{\frac{1}{4}}\notag 
        \\
        &\le C \Sob{\partial_{y} \overline{\omega}}{L^2} \Sob{u}{L^\infty}{\Tr[- \Delta \PP_\tN(s)]}^{\frac{7}{8}}.
    \end{align}
Upon taking time averages, then applying H\"older's inequality, we obtain
    \begin{align}\label{E:tracebound3:final}
        \lbn \Tr[\overline{\BB}_1\QQ_{\tN}\AA_0\PP_{\tN}]  \rbn 
         \geq-C\left(\sup_{s}\Sob{u(s)}{L^\infty}\Sob{\partial_{y} \overline{\omega}(s)}{L^2}^{\frac{3}{4}}\right) \lbn \Sob{\partial_{y} \overline{\omega}}{L^2}^2 \rbn^{\frac{1}{8}}\lbn \Tr[-\Delta \PP_\tN] \rbn^{\frac{7}{8}}.
    \end{align}
    
For the third and last term, $III$, in \eqref{E:tracesplit1}, we have
    \begin{align}\label{E:tracebound4}
        |III|
        &=\left|\sum_{j=1}^{\tN}\frac i2\left(e^{sL_{\veps}}\partial_y \overline{\omega}e^{-sL_\veps}\QQ_{\tN}(s)B_\veps(t)(\phi_j(s),\eta(s)),\phi_j(s)\right)\right|\notag
        \\
        &=\left|\sum_{j=1}^{\tN}\frac i2\left(e^{-sL_\veps}\QQ_{\tN}(s)e^{sL_{\veps}}B( e^{-sL_\veps}\phi_j(s),\omega),\partial_y \overline{\omega}e^{-sL_\veps}\phi_j(s)\right)\right|\notag
        \\
        &\le \sum_{j=1}^{\tN}\Sob{e^{-sL_\veps}\QQ_{\tN}(s)e^{sL_{\veps}}B( e^{-sL_\veps}\phi_j(s),\omega)}{L^2}\Sob{\partial_y \overline{\omega}e^{-sL_\veps}\phi_j(s)}{L^2}.
    \end{align}
We estimate this term in the manner of \eqref{E:CS1}-\eqref{E:basis-Linfinity} to obtain
\begin{align}\label{E:tracebound4:final}
    &\langle\Tr[\overline{\BB}_1(s)\QQ_{\tN}(s)\BB(s)\PP_{\tN}(s)]\rangle \notag\\
    &\quad\ge -C\left(\sup_{s}\Sob{\partial_y \overline{\omega}}{L^\infty}\right)\left\langle \Sob{\nabla {\omega}}{L^2}^2 \right\rangle^{1/2} \left \langle\Tr[-\Delta \PP_\tN(s)]\right\rangle^{1/4}\left(1+\log\left \langle\Tr[- \Delta \PP_\tN(s)]\right\rangle\right)^{1/2}.
\end{align}
Collecting the estimates \eqref{E:tracebound2:final}, \eqref{E:tracebound3:final}, and \eqref{E:tracebound4:final}, yields \eqref{E:tracebound234:final}, as desired.
\end{proof}

We are left to treat one last term: $\lbn \Tr[\PP_\tN(s)\overline{\BB}_{2}(s)\PP_\tN(s)]\rbn$.

\subsubsection{Treating $\lbn \Tr[\PP_\tN(s)\overline{\BB}_{2}(s)\PP_\tN(s)]\rbn$}

We prove the following:

\begin{Lem}\label{lem:tracebound5:final}
    \begin{align}\label{E:tracebound5:final}
        \lbn \Tr[\PP_{\tN}(s)\overline{\BB}_{2}(s)\PP_{\tN}(s)]  \rbn
        &\ge -C\left (\sup_{s}\Sob{\partial_{yy} \overline{\omega}}{L^2}^{\frac{3}{4}}\right)  \lbn \Sob{\partial_{yy} \overline{\omega}}{L^2}^2\rbn ^{\frac{1}{8}}\lbn \Tr[-\Delta \PP_\tN] \rbn^{\frac{7}{8}}-C \Sob{\overline{f}}{L^2}\lbn \Tr[-\Delta \PP_\tN] \rbn^{\frac{7}{8}}\notag
        \\
        &\quad-C\left(\sup_{s}\Sob{u}{L^\infty}\Sob{\partial_{y} \overline{\omega}}{L^2}^{\frac{3}{4}}\right) \lbn \Sob{\partial_{y} \overline{\omega}}{L^2}^2 \rbn^{\frac{1}{8}}\lbn \Tr[-\Delta \PP_\tN] \rbn^{\frac{7}{8}}.
    \end{align}
\end{Lem}

\begin{proof}
Integrating by parts, we have
    \begin{align}\label{E:tracebound5}
        \Tr[\PP_\tN(s)\overline{\BB}_{2}(s)\PP_{\tN}(s)]&=\sum_{j=1}^{\tN}\frac i2\left(e^{sL_{\veps}}\partial_s\partial_y \overline{\omega}e^{-sL_\veps}\phi_j(s),\phi_j(s)\right)\notag
        \\
        &=\sum_{j=1}^{\tN}\frac i2\left(\partial_{s}\overline{\omega},\partial_y e^{-sL_\veps}\phi_j(s) e^{-sL_\veps}\phi_j(s)\right).
    \end{align}

Now from \eqref{eq:NSE:rotatingvf}, we have
    \begin{align*}
        \partial_s \overline{\om}=-\overline{B}({\om},{\om})+\partial_{yy} \overline{\om}+\overline{f}.
    \end{align*}
Using \eqref{E:tracebound5} and estimating similar to \eqref{E:tracebound2}, we obtain
    \begin{align*}
        &\left|\Tr[\PP_\tN(s)\overline{\BB}_{2}(s)\PP_{\tN}(s)]\right|\notag
        \\
        &\le C(\Sob{u}{L^\infty} \Sob{\nabla\om}{L^2}+\Sob{\partial_{yy}\overline{\om}}{L^2}+\Sob{\overline{f}}{L^2})\left(\sum_{j=1}^{\tN} \Sob{\nabla\phi_j(s)}{L^2}^2\right)^{\frac{3}{4}}\left(\sum_{j=1}^{\tN} \Sob{\phi_j(s)}{L^2}^2\right)^{\frac{1}{4}}.
    \end{align*}
Upon taking time average, and applying Holder and Jensen's inequality, we obtain
    \begin{align}
        \lbn \Tr[\PP_{\tN}(s)\overline{\BB}_{2}(s)\PP_{\tN}(s)]  \rbn \ge& -C\left (\sup_{s}\Sob{\partial_{yy} \overline{\omega}(s)}{L^2}^{\frac{3}{4}}\right)  \lbn \Sob{\partial_{yy} \overline{\omega}}{L^2}^2\rbn ^{\frac{1}{8}}\lbn \Tr[-\Delta \PP_\tN] \rbn^{\frac{7}{8}}-C \Sob{\overline{f}}{L^2}\lbn \Tr[-\Delta \PP_\tN] \rbn^{\frac{7}{8}}\notag 
        \\&-C\left(\sup_{s}\Sob{u(s)}{L^\infty}\Sob{\nabla \omega(s)}{L^2}^{\frac{3}{4}}\right) \lbn \Sob{\nabla\omega}{L^2}^2 \rbn^{\frac{1}{8}}\lbn \Tr[-\Delta \PP_\tN] \rbn^{\frac{7}{8}},\notag
    \end{align}
which is precisely \eqref{E:tracebound5:final}, and we are done.
\end{proof}

\subsection{Culmination of the Estimates} 

We are finally in a position to prove \cref{thm:main}. Let us first summarize the estimates above. Denote 
    \[
        \mathscr{N}^2= \lbn \Tr[-\Delta \PP_\tN] \rbn\ge C\tN^2.
    \]
Upon combining the results \cref{lem:tracebound1:final}, \cref{lem:Bbar:decomposition}, 
\cref{lem:tracebound234:final}, and \cref{lem:tracebound5:final}, we arrive at
    \begin{align}\label{est:tracesummary}
        \left\langle \Tr[\AA\PP_\tN]\right \rangle  &\ge \mathscr{N}^2-C\left\langle \Sob{\nabla \widetilde{\omega}}{L^2}^2 \right\rangle^{\frac{1}{2}}\mathscr{N}^{\frac{1}{2}}\left(1+\log \mathscr{N}\right)^{\frac{1}{2}}\notag
        \\
        &\quad-C\varepsilon\left(\sup_{s}\Sob{u(s)}{L^\infty}\Sob{\nabla \omega(s)}{L^2}^{\frac{3}{4}}\right) \lbn \Sob{\nabla \omega}{L^2}^2 \rbn^{\frac{1}{8}}\mathscr{N}^{\frac{7}{4}}\notag 
        \\
        &\quad-C\varepsilon\left(\sup_{s}\Sob{\partial_y \overline{\omega}}{L^\infty}\right)\left\langle \Sob{\nabla {\omega}}{L^2}^2 \right\rangle^{\frac{1}{2}} \mathscr{N}^{\frac{1}{2}}\left(1+\log\mathscr{N}\right)^{\frac{1}{2}}-C\varepsilon\left(\sup_{s}\Sob{\partial_y \overline{\omega}(s)}{L^\infty}\right)\mathscr{N}^2\notag
        \\
        &\quad-C\varepsilon\left (\sup_{s}\Sob{\partial_{yy} \overline{\omega}(s)}{L^2}^{\frac{3}{4}}\right)  \lbn \Sob{\partial_{yy} \overline{\omega}}{L^2}^2\rbn ^{\frac{1}{8}}\mathscr{N}^{\frac{7}{4}}-C \varepsilon\Sob{\overline{f}}{L^2}\mathscr{N}^{\frac{7}{4}}
    \end{align}
To estimate $\lbn \Sob{\nabla \widetilde{\om}}{L^2}^2\rbn$, we replicate the analysis as given in \cite{MustafaDjoko2011} but for a time independent forcing $f$. We state the estimate below and postpone its proof until \cref{sect:app:A}.

\begin{Lem}\label{lem:om:timeavg} Let $\om$ be a solution of \eqref{eq:NSE:rotatingv} on $\Acal^\veps$ with non-zonal component denoted by $\widetilde{\om}$. Then
    \begin{align*}
        \sup_{s}\Sob{\widetilde{\om}(s)}{L^2}+\lbn \Sob{\nabla\widetilde{\om}}{L^2}^2\rbn^{\frac{1}{2}} \le C{\varepsilon}^{\frac{1}{2}} \Gr^{\frac{9}{4}}(1+\log \Gr)^{\frac{1}{4}},
    \end{align*}
for some constant $C$ independent of $\veps$.
\end{Lem}

We are now ready to prove the main theorem, \cref{thm:main}.

\subsection*{Proof of \cref{thm:main}}

By \eqref{est:om:sobolev:Gr}, we have
    \[
        \sup_{s}\Sob{\nabla \om(s)}{L^2}\le C \Gr^{2}, \quad \sup_{s}\Sob{\Delta\om(s)}{L^2}\le C \Gr^{3}.
    \]
On the other hand, by \eqref{est:time:om} and \eqref{est:time:om:Lap}, we have
    \[
        \lbn \Sob{\nabla \om}{L^2}^2\rbn^{\frac{1}{2}}\le \Gr,\quad  \lbn \Sob{\Delta \om}{L^2}^2\rbn^{\frac{1}{2}}\le C\Gr^2 .
    \]
For estimating $\Sob{\partial_y \overline{\omega}}{L^\infty}$, we invoke the 1D Agmon inequality to get
    \begin{align*}
        \sup_{s}\Sob{\partial_y \overline{\omega}(s)}{L^\infty}\le C\sup_{s}\Sob{\overline{\omega}(s)}{H^1}^{\frac{1}{2}}\Sob{\overline{\omega}(s)}{H^2}^{\frac{1}{2}}\le C \Gr^{\frac{5}{2}},
    \end{align*}
    Similarly, by the 2D Br\'ezis-Gallouet inequality (see \cite{MustafaDjoko2011}), we have
    \begin{align*}
        \sup_{s}\Sob{u(s)}{L^\infty}\le C\sup_{s}\Sob{\om(s)}{L^2}\left[1+\log\left(\frac{\Sob{\Delta \om(s)}{L^2}}{\Sob{\nabla \om(s)}{L^2}}\right)\right]^{\frac{1}{2}}\le C \Gr (1+\log \Gr)^{\frac{1}{2}}.
    \end{align*}
Upon applying the above bounds and \cref{lem:om:timeavg} in \eqref{est:tracesummary}, we obtain
    \begin{align*}
        \left\langle \Tr[\AA\PP_\tN]\right \rangle 
        &\ge \mathscr{N}^2-C{\varepsilon}^{\frac{1}{2}} \Gr^{\frac{9}{4}}(1+\log \Gr)^{\frac{1}{4}}\mathscr{N}^{\frac{1}{2}}(1+\log\mathscr{N})^{\frac{1}{2}}\notag
        \\
        &\quad -C\veps \Gr^{\frac{11}{4}}(1+\log\Gr)^{\frac{1}{2}}\mathscr{N}^{\frac{7}{4}}\notag
        \\
        &\quad-C\varepsilon \Gr^{\frac{7}{2}}\mathscr{N}^{\frac{1}{2}}\left(1+\log\mathscr{N}\right)^{\frac{1}{2}}-C\varepsilon\Gr^{\frac{5}{2}}\mathscr{N}^2\notag
        \\
        &\quad-C\varepsilon\Gr^{\frac{11}{4}}\mathscr{N}^{\frac{7}{4}}-C\veps\Gr\mathscr{N}^{\frac{7}{4}}\notag
        \\
        &= \left(1-C\veps\Gr^{\frac{5}{2}}\right)\mathscr{N}^2\notag
        \\
        &\quad-C\veps^{\frac{1}{2}}\Gr^{\frac{9}{4}}\left[(1+\log\Gr)^{\frac{1}{4}}+\veps^{\frac{1}{2}}\Gr^{\frac{5}{4}}\right]\mathscr{N}^{\frac{1}{2}}(1+\log\mathscr{N})^{\frac{1}{2}}\notag
        \\
        &\quad-C\veps \Gr\left[\Gr^{\frac{7}{4}}(1+\log\Gr)^{\frac{1}{2}}+\Gr^{\frac{7}{4}}+1\right]\mathscr{N}^{\frac{7}{4}}.
    \end{align*}
We assume that $\varepsilon$ satisfies the following  condition
    \begin{align}\label{cond:eps:Gr}
        C\varepsilon \Gr^{5/2}<\frac{1}2, 
    \end{align}
Using this, along with the fact that $\Gr\geq1$, we obtain
    \begin{align}\label{est:trace:final}
        \left\langle \Tr[\AA\PP_\tN]\right \rangle  \ge &\frac{1}2 \mathscr{N}^2-C{\varepsilon}^{\frac{1}{2}}\Gr^{\frac{9}{4}}(1+\log\Gr)^{\frac{1}{4}}\mathscr{N}^{\frac{1}{2}}\left(1+\log \mathscr{N}\right)^{\frac{1}{2}}-C\varepsilon\Gr^{\frac{11}{4}}(1+\log\Gr)^{\frac{1}{2}}\mathscr{N}^{\frac{7}{4}}
    \end{align}

Now, to obtain an upper bound estimate on the smallest value of $\tN$ such that $\left\langle \Tr[\AA\PP_\tN]\right \rangle >0$, we find the cross-over points of the right-hand side of \eqref{est:trace:final} as a function of $\mathscr{N}$, that is, the points $\mathscr{N}$ such that
    \[  
        \mathscr{N}^2 \sim \varepsilon \Gr^{\frac{11}{4}}(1+\log\Gr)^{\frac{1}{2}}\mathscr{N}^{\frac{7}{4}}\quad\text{and}\quad \mathscr{N}^2 \sim {\varepsilon}^{\frac{1}{2}}\Gr^{\frac{9}{4}}(1+\log\Gr)^{\frac{1}{4}}\mathscr{N}^{\frac{1}{2}}\left(1+\log \mathscr{N}\right)^{\frac{1}{2}}.
    \]
This yields
    \[  
        \mathscr{N} \sim \varepsilon^4 \Gr^{11} (1+\log\Gr)^{2}\quad\text{and}\quad \mathscr{N} \sim \varepsilon^{\frac{1}{3}}\Gr^{\frac{3}{2}}(1+\log\Gr)^{\frac{1}{2}}.
    \]
Thus, we have
    \begin{align*}
          \mathscr{N}&\sim \max \left\{\varepsilon^4 \Gr^{11} (1+\log\Gr)^{2}, \varepsilon^{1/3}\Gr^{3/2}(1+\log\Gr)^{1/2}\right\},
    \end{align*}
Observe that 
\[\varepsilon^4 \Gr^{11} (1+\log\Gr)^{2}<\varepsilon^{1/3}\Gr^{3/2}(1+\log\Gr)^{1/2}\quad \text{if}\quad \varepsilon<\Gr^{-57/22}(1+\log \Gr)^{-9/22}.\]
Upon applying \eqref{E:A0}, we finally obtain
    \begin{align}\notag
        \tN\leq\mathscr{N}\lesssim 
        \begin{cases}
           \varepsilon^{1/3}\Gr^{3/2}(1+\log\Gr)^{1/2} \quad &\text{if}\quad \varepsilon \lesssim\Gr^{-57/22}(1+\log \Gr)^{-9/22}\\
          \varepsilon^4 \Gr^{11} (1+\log\Gr)^{2}   \quad &\text{if}\quad \Gr^{-57/22}(1+\log \Gr)^{-9/22} \lesssim \veps \lesssim \Gr^{-5/2},
        \end{cases},
    \end{align}
as claimed.
\hfill \qed

\begin{Rmk}\label{rmk:smallness}
We point out that the smallness condition \eqref{cond:eps:Gr} is determined by the need to control $\sup_s\|\partial_y\overline{\omega}(s)\|_{L^\infty}$. This appears to be unavoidable in our approach. In terms of Grashof number, this term is bounded by $\Gr^{5/2}$.

On the other hand, observe that 
    \[\varepsilon^4 \Gr^{11} (1+\log\Gr)^{2}< \Gr^{2/3}(1+\log \Gr)^{1/3}\quad \text{if}\quad \veps<\Gr^{-31/12}(1+\log \Gr)^{-5/12}.\]    
In particular, the number $31/12$ is the exponent that determines the cut-off when our attractor dimension estimates match the classical dimension estimate of Constantin-Foias-Temam (see \eqref{cond:regime:2}). Notice that $5/2<31/12$. In principle, it is possible that one can improve upon the bottleneck exponent of $5/2$, in which case the exponent $31/12$ may change.
\end{Rmk}

\section{A description of the asymptotic behavior of $\om^\veps$}\label{sect:asymptotic}

In the final section, we formalize a heuristic of Al-Jaboori \& Wirosoetisno \cite{MustafaDjoko2011} for understanding the asymptotic behavior of solutions to 2D rotating NSE as $\veps\goesto0$, and provide an explicit decomposition of solutions to the 2D rotating NSE that encodes this information. In particular, we prove \cref{thm:main} and \cref{thm:continuity}.

Observe that we may decompose \eqref{eq:NSE:rotatingv} into its corresponding to non-zonal and zonal components to obtain the coupled system
    \begin{align*}
        \frac{\partial \overline{\omega}}{\partial t}+\overline{B}(\omega, \omega)&=\overline{f}+\nu \Delta \overline{\omega}
        \\
        \frac{\partial \widetilde{\omega}}{\partial t}+\widetilde{B}(\omega, \omega)&=\widetilde{f}+\nu \Delta \widetilde{\omega}-L_\veps\widetilde{\omega}.
    \end{align*}

It is shown in \cite[Theorem 3.1]{MustafaDjoko2011} that for all $m=0, 1,2, \dots$, there exists $T_m$, independent of $\veps$, such that
    \begin{align}\label{est:nonzonal:bounds}
        \sup_{t\geq T_m}\Sob{\widetilde{\om}(t)}{H^m}=O(\sqrt{\veps}).
    \end{align}
It follows that
    \[
        \overline{B}(\omega, \omega)=\overline{B}(\widetilde{\omega}, \widetilde{\omega})+\overline{B}(\widetilde{\omega}, \overline{\omega})+\overline{B}(\overline{\omega}, \widetilde{\omega})\xrightarrow{L^2} 0 \quad \text{as}\quad \varepsilon \rightarrow 0. 
    \]
This observation leads us to prove the following.

\begin{Lem}\label{lem:limiting:system}
For each $\veps>0$, there exists $T_*(\veps)>0$ such that
    \begin{align}\label{est:omeps:ombar}
        \sup_{t\geq T_*(\veps)}\|\om^\veps(t)-\overline{\om}(t)\|_{L^2}\leq O(\sqrt{\veps})
    \end{align}
where $\overline{\om}$ satisfies
    \begin{align}\label{E:heat:limit}
        \frac{\partial \overline{\omega}}{\partial t}-\nu \Delta \overline{\omega}=\overline{f}.
    \end{align}
In particular   
    \begin{align}\notag
        \lim_{\veps\goesto0^+}\sup_{t\geq T_1'(\veps)}\|\om^\veps(t)-\overline{\om}(t)\|_{L^2}=0.
    \end{align}
\end{Lem}
\begin{proof}
Let $\overline{\ze}^\veps:=\overline{\om}^\veps-\overline{\om}$. Then
    \begin{align}\notag
        \frac{\bdy\overline{\ze}^\veps}{\bdy t}-\De\overline{\ze}^\veps=-\overline{B}(\om^\veps,\om^\veps).
    \end{align}
Observe that
    \begin{align}\label{est:zetabar:predecomp}
        \overline{B}(\om^\veps,\om^\veps)=\overline{B}(\overline{\om}^\veps,\widetilde{\om}^\veps)+\overline{B}(\widetilde{\om}^\veps,\overline{\om}^\veps)+\overline{B}(\widetilde{\om}^\veps,\widetilde{\om}^\veps)=\overline{B}(\widetilde{\om}^\veps,\widetilde{\om}^\veps).
    \end{align}
Then
    \begin{align}
        |(\overline{B}(\widetilde{\om}^\veps,\widetilde{\om}^\veps),\overline{\ze}^\veps)|&=|({B}(\widetilde{\om}^\veps,\overline{\ze}^\veps),\widetilde{\om}^\veps)|\notag
        \\
        &\leq C\|(-\De)^{-1/2}\widetilde{\om}^\veps\|_{L^4}\|\widetilde{\om}\|_{L^4}\|\nabla\overline{\ze}^\veps\|_{L^2}\notag
        \\
        &\leq C\|\widetilde{\om}^\veps\|_{L^2}^{1/2}\|(-\De)^{-1/2}\widetilde{\om}^\veps\|_{L^2}^{1/2}\|\nabla\widetilde{\om}^\veps\|^{1/2}\|\widetilde{\om}^\veps\|_{L^2}^{1/2}\|\nabla\overline{\ze}^\veps\|_{L^2}\notag
        \\
        &\leq C\|\widetilde{\om}^\veps\|_{H^1}^2\|\nabla\overline{\ze}^\veps\|_{L^2}\notag
        \\
        &\leq C\|\widetilde{\om}^\veps\|_{H^1}^4+\frac{1}2\|\nabla\overline{\ze}^\veps\|_{L^2}^2.\notag
    \end{align}
Then    
    \begin{align}
        \frac{d}{dt}\|\overline{\ze}^\veps\|_{L^2}^2+\|\nabla\overline{\ze}^\veps\|_{L^2}^2\leq C\|\widetilde{\om}^\veps\|_{H^1}^4\notag
    \end{align}
By the Gr\"onwall inequality, it follows that
    \begin{align}
        \|\overline{\ze}^\veps(t)\|_{L^2}^2&\leq e^{-(t-T_1)}\|\overline{\ze}^\veps(T_1)\|_{L^2}^2+\int_{T_1}^te^{-(t-s)}\|\widetilde{\om}^\veps(s)\|_{H^1}^4ds\notag
        \\
        &\leq e^{-(t-T_1)}\|\overline{\ze}^\veps(T_1)\|_{L^2}^2+C\sup_{t\geq T_1}\|\widetilde{\om}^\veps(t)\|_{H^1}^4,\label{est:zetabar:bound}
    \end{align}
holds for all $t\geq T_1$, independent of $\veps$, for all $\veps$. 

Since $T_1$ is independent of $\veps$, we may pass to the limit $\veps\goesto0$. Then
    \begin{align}\notag
        \limsup_{\veps\goesto0} \|\overline{\ze}^\veps(t)\|_{L^2}^2&\leq e^{-(t-T_1)}\limsup_{\veps\goesto0}\|\overline{\ze}^\veps(T_1)\|_{L^2}^2,
    \end{align}
holds for all $t\geq T_1$. Finally, since 
    \begin{align}\notag
        \limsup_{\veps\goesto0}\|\overline{\ze}^\veps(T_1)\|_{L^2}^2\leq \limsup_{\veps\goesto0}\|\overline{\om}^\veps(T_1)\|_{L^2}^2+\|\overline{\om}(T_1)\|_{L^2}^2\leq O(1),
    \end{align}
it follows that
    \begin{align}\label{est:zetabar:zero}
         \lim_{t\goesto\infty}\limsup_{\veps\goesto0} \|\overline{\ze}^\veps(t)\|_{L^2}^2&=0.
    \end{align}
Upon combining \eqref{est:nonzonal:bounds}, \eqref{est:zetabar:bound}, and \eqref{est:zetabar:zero}, we deduce that for each $\veps$, there exists $T_*(\veps)$ such that
    \begin{align}\notag
        \sup_{t\geq T_*(\veps)}\|\overline{\om}^\veps(t)-\overline{\om}(t)\|_{L^2}\leq O(\veps).
    \end{align}
In conjunction with \eqref{est:nonzonal:bounds}, we may now conclude \eqref{est:omeps:ombar}.
\end{proof}

Observe that the above result implies that we can define the limit as of $\om^\veps$ as $\veps\goesto0$ as $\om^0_*:=\overline{\om}_*$. However, it is not clear what the relation of $\overline{\om}_*$ is to the dynamics of the limiting infinite-rotation system. It is possible that the limiting system possesses a non-trivial global attractor, and that $\overline{\om}_*$ represents only one of its steady states. This interesting issue will be reserved for a future study. We nevertheless can make a simple observation. For $\varepsilon$ sufficiently small, let $\om^\varepsilon_*$ be the unique steady state for \eqref{eq:NSE:rotatingv}, and let $\om^0_*$ be the unique steady state of \eqref{E:heat:limit}. We then prove the convergence result for attractors $\Acal_\varepsilon$ as $\varepsilon \rightarrow 0$ as stated below.

\begin{Lem}\label{lem:attractors:cty}
Let $\om^0_*:=\overline{\om}_*$. Then
    \[
        \om^\varepsilon_* \rightarrow \om^0_*\quad\text{in}\ H^1\quad \text{as}\quad \varepsilon \rightarrow 0^+.
    \]
In particular $\lim_{\veps\goesto0^+}\dist(\Acal^\veps,\overline{\Acal})=0$, where $\overline{\Acal}$ denotes the global attractor of \eqref{eq:heat}.
\end{Lem}

\begin{proof}
Let $\ze^\veps=\om^\veps_*-\om^0_*$. Then $\ze^\veps$ satisfies 
            \begin{align*}
                B(\ze^\veps,\ze^\veps)+B(\om^0_*, \ze^\veps)+B(\ze^\veps, \om^0_*)=\nu\De \ze^\veps+L_\veps\ze^\veps.
            \end{align*}
Upon taking inner product in $L^2$ with $\ze^\veps$, we obtain
        \begin{align*}
            \nu\|\nabla \ze^\veps\|_{L^2}^2=-\lb B(\widetilde{\ze}^\veps, \om^0_*),\widetilde{\ze}^\veps \rb=\lb B(\widetilde{\om}^\veps_{*},\om^0_*),\widetilde{\om}^\veps_*\rb \le C \Sob{\nabla^{-1}\widetilde{\om}^\veps_{*}} {L^\infty}\Sob{\nabla \om^0_*}{L^2}\Sob{\widetilde{\om}^\veps_*}{L^2}\rightarrow 0
        \end{align*}
    as $\varepsilon \rightarrow 0$.
\end{proof}

Finally, we prove \cref{thm:decomposition}, which is essentially just a collective summary of the results \cref{lem:limiting:system}, \cref{lem:attractors:cty}. 

\begin{proof}[Proof of \cref{thm:decomposition}]
By \cref{lem:limiting:system} it follows that the solution of rotating NSE can be decomposed as
    \begin{align}\notag
        \om^\veps(t)
        =\overline{\om}(t)+\widetilde{\om}^\veps(t)+\overline{\ze}^\veps(t)\goesto \overline{\om}_*+\widetilde{\om}^\veps_*+O(\veps)\quad\text{as}\ t\goesto\infty,
    \end{align}
whenever $\veps$ is sufficiently small. Hence
    \begin{align}\notag
        \lim_{\veps\goesto0^+}\om^\veps_*=\overline{\om}_*\quad\text{in}\ L^2,
    \end{align}
By \cref{lem:attractors:cty}, this convergence can be upgraded to $H^1$.
\end{proof}

\subsection*{Acknowledgments} The authors would like to thank the referees for their valuable suggestions to improve the manuscript. A.F. was supported in part by the National Science Foundation through DMS 2206493. V.R.M. was in part supported by the National Science Foundation through DMS 2213363 and DMS 2206491, as well as the Mary P. Dolciani Halloran Foundation.

\appendix

\section{Proof of \cref{lem:om:timeavg}}\label{sect:app:A}

\begin{proof}
The proof of \cref{lem:om:timeavg} consists of several steps. First, we estimate $\lbn \Sob{\nabla \om} {L^2}^2\rbn$ in terms of Grashof number $\Gr$, which is then used to obtain a bound for $\lbn \Sob{\Delta \om} {L^2}^2\rbn$. Then, we write the evolution equation for $\Sob{\widetilde{\omega}}{L^2}$ and decompose the nonlinear term into several terms using the product rule in time as done in \eqref{E:IBP}. Finally, we estimate all the terms using corresponding estimates for norms of $\omega$. 

For convenience, we recall definitions of $\eta, f, g$ from \cref{sec:prelim}.
\[\eta=e^{tL_\veps}\omega,\quad f=\Gr (\nabla^{\perp}\cdot F),\quad g=e^{tL_\veps}f.\]
Multiply \eqref{eq:NSE:rotatingv} by $\om$, integrate in space and take the time average to obtain
\[\lbn \Sob{\nabla{\om}}{L^2}^2\rbn=\lbn (\om,f)\rbn.\]
Note that $(L_\eps \om, \om)=0$ by \eqref{eq:skew}. Using integration by parts and applying the Cauchy-Schwarz inequality, we obtain
\[\lbn \Sob{\nabla{\om}}{L^2}^2\rbn\le \lbn \Sob{\nabla{\om}}{L^2}\rbn\Gr\Sob{F}{L^2}\le \lbn \Sob{\nabla{\om}}{L^2}^2\rbn^{1/2}\Gr.\]
Thus, we have
    \begin{align}\label{est:time:om}
        \lbn \Sob{\nabla \om}{L^2}^2\rbn^{1/2}\le \Gr.
    \end{align}
Next, to estimate $\lbn \Sob{\Delta \om}{L^2}^2\rbn$, we multiply \eqref{eq:NSE:rotatingv} by $\Delta\om$, integrate in space and take the time average to obtain
    \[
        \lbn \Sob{\Delta{\om}}{L^2}^2\rbn =-\lbn (B(\om,\om),\Delta \om)\rbn+ \lbn (\Delta \om,f)\rbn.
        \]
Integrating by parts and applying Ladyzhenskaya's inequality and the Cauchy-Schwarz inequality, we obtain
    \begin{align*}
        \lbn \Sob{\Delta{\om}}{L^2}^2\rbn& \le \lbn |(B(\nabla \om,\om),\nabla \om)|\rbn+ \lbn |(\Delta \om,f)|\rbn \\& \le C\lbn \Sob{\nabla \om}{L^4}^2 \Sob{\om}{L^2}\rbn + \lbn \Sob{\Delta{\om}}{L^2}\rbn\Sob{f}{L^2}\\ &\le C\sup_{s}\{\Sob{\om}{L^2}\}\lbn \Sob{\nabla \om}{L^2} \Sob{\Delta \om}{L^2} \rbn+\lbn \Sob{\Delta{\om}}{L^2}\rbn\Sob{f}{L^2}
        \\
        &\le C \sup _{s}{\Sob{\om}{L^2}}\lbn \Sob{\Delta{\om}}{L^2}^2\rbn^{1/2} \lbn \Sob{\nabla{\om}}{L^2}^2\rbn^{1/2}+  \lbn \Sob{\Delta{\om}}{L^2}^2\rbn^{1/2}\Sob{f}{L^2}
        \\
        &\le C( \Gr^2+\Gr)\lbn \Sob{\Delta{\om}}{L^2}^2\rbn^{1/2}.
    \end{align*}
Thus, we have
    \begin{align}\label{est:time:om:Lap}
        \lbn \Sob{\Delta \om}{L^2}^2\rbn^{1/2}\le C\Gr^2.
    \end{align}
Next, we multiply \eqref{eq:NSE:rotatingv} by $\widetilde{\om}$ and integrate in space to obtain
    \begin{align}\label{eq:NSE:rotatingv:non-zonal}
        \frac{1}{2}\frac{d}{ds} \Sob{\widetilde{\om}}{L^2}^2+\Sob{\nabla \widetilde{\om}}{L^2}^2+(B(\om,\om),\widetilde{\om})=(f,\widetilde{\om})-(L_\veps\om, \widetilde{\om}).
    \end{align}
By writing $\om=\overline{\om}+\widetilde{\om}$ and using \eqref{eq:skew}, we have 
    \[
        (L_\veps \om, \widetilde{\om})=(L_\veps\widetilde{\om}, \widetilde{\om})=0.
    \]
Similarly, we have
    \begin{align*}
        (B(\om,\om),\widetilde{\om})=-(B(\widetilde{\om},\widetilde{\om}),\overline{\om})=-(B(e^{-sL_\veps}\widetilde{\eta},e^{-sL_\veps}\widetilde{\eta}),\overline{\omega}).
    \end{align*} 
Applying the Plancherel's theorem and using the product rule in time as described in the analysis preceding \eqref{E:IBP}, we obtain
    \begin{align*}
        &(B(e^{-sL_\veps}\widetilde{\eta},e^{-sL_\veps}\widetilde{\eta}),\overline{\omega})\\&=-\frac{i\varepsilon}{2} \partial_s \left((e^{-sL_\veps}\widetilde{\eta})^2,\partial_y \overline{\omega}\right)+i\varepsilon \left((e^{-sL_\veps}\partial_s \widetilde{\eta})(e^{-sL_\veps}\widetilde{\eta}),\partial_y \overline{\omega}\right)
    +\frac{i\varepsilon}{2} \left((e^{-sL_\veps}\widetilde{\eta})^2,\partial_s \partial_y \overline{\omega}\right)
        \\
        &=I_1+I_2+I_3.
    \end{align*}

Applying the Plancherel's theorem, we have
    \begin{align*}
        (f,\widetilde{\om})&=4\pi^2\sum_{k \in \mathbb{Z}^2}\mathcal{F}_k(f)\overline{\mathcal{F}_k({\omega})}e^{-i\lambda_k s/\varepsilon}\notag
        \\
        &=i4\pi^2 \veps\sum_{k \in \mathbb{Z}^2, \lam_k \ne 0}\frac{1}{\lambda_k}\partial_s\left(\mathcal{F}_k(f)\overline{\mathcal{F}_k({\omega})}e^{-i\lambda_k s/\varepsilon}\right)-i4\pi^2 \veps\sum_{k \in \mathbb{Z}^2, \lam_k \ne 0}\frac{1}{\lambda_k}\mathcal{F}_k(f)\overline{\mathcal{F}_k({\partial_s \omega})}e^{-i\lambda_k s/\varepsilon}\notag
        \\
        &=i\veps \partial_s(\check{f},\widetilde{\omega})-i\veps (\check{f},\partial_s\widetilde{\omega})\\
        &=J_1 +J_2,
    \end{align*}
where $\check{f}$ is defined by
\begin{align*}
   \check{f}_k= \begin{cases}
        \frac{f_k}{\lam_k}\quad \text{if}\, \lam_k\ne 0,\\
        0 \quad \text{if}\, \lambda_k=0.
    \end{cases}
\end{align*}
Observe that by using uniform-in-time bounds in \eqref{est:om:sobolev:Gr}, we conclude that
    \begin{align*}
        \lbn I_1 \rbn, \lbn J_1 \rbn=0.
    \end{align*}
From \eqref{eq:NSE:rotatingvf:eta}, we have 
    \begin{align*}
        e^{-sL_\veps}\partial_s \widetilde{\eta}&=e^{-sL_\veps}(-\widetilde{B_\veps(t)}(\eta,\eta)+\widetilde{g
}+\Delta \til {\eta
})
    \\
        &=-\widetilde{B}(\om,\om)+\widetilde{f}+\Delta \widetilde{\om}.
    \end{align*}
Thus, we have
    \[
        I_2=-i\varepsilon(\widetilde{B}(\om,\om)\widetilde{\om},\partial_y \overline{\om})+i\varepsilon(\widetilde{f}\widetilde{\om},\partial_y \overline{\om})+i\varepsilon((\Delta\widetilde{\om})\widetilde{\om},\partial_y \overline{\om}).
    \]
We estimate $I_2$ by applying Holder inequality, Agmon inequality, and Br\'ezis-Gallouet inequality
    \begin{align}
        |I_2|\le & C\varepsilon \Sob{\partial_y \overline{\om}}{L^2_y}\Sob{u}{L^\infty}\int \Sob{\nabla \om}{L^2_y}\Sob{\widetilde{\om}}{L^\infty_y}\, dx+C\varepsilon \Sob{\partial_y \overline{\om}}{L^2_y} \int \Sob{\widetilde{f}}{L^2_y}\Sob{\widetilde{\om}}{L^\infty_y}\,dx \notag
        \\
        &+C\varepsilon \Sob{\partial_y \overline{\om}}{L^2_y} \int \Sob{\Delta \widetilde{\om}}{L^2_y}\Sob{\widetilde{\om}}{L^\infty_y}\,dx \notag
        \\
        \le & C\varepsilon \Sob{\partial_y \overline{\om}}{L^2}\Sob{u}{L^\infty}\int \Sob{\nabla \om}{L^2_y}^{3/2}\Sob{\om}{L^2_y}^{1/2}\, dx+C\varepsilon \Sob{\partial_y \overline{\om}}{L^2} \int \Sob{{f}}{L^2_y}\Sob{\nabla \om}{L^2_y}^{1/2}\Sob{\om}{L^2_y}^{1/2}\,dx \notag
        \\
        &+C\varepsilon \Sob{\partial_y \overline{\om}}{L^2} \int \Sob{\Delta {\om}}{L^2_y}\Sob{\nabla \om}{L^2_y}^{1/2}\Sob{\om}{L^2_y}^{1/2}\,dx \notag
        \\
        \le & C\varepsilon\left(\Sob{\om}{L^2}^{{3}/{2}}\log\left(\frac{\Sob{\Delta \om}{L^2}}{c\Sob{\nabla \om}{L^2}}+1\right)^{1/2}\Sob{ \nabla \om}{L^2}^{5/2}+\Sob{f}{L^2}\Sob{\om}{L^2}^{{1}/{2}}\Sob{\nabla \om}{L^2}^{{3}/{2}}\right. \notag
        \\ 
        &\left.\hspace{15 em}+\Sob{\Delta\om}{L^2}\Sob{\om}{L^2}^{{1}/{2}}\Sob{\nabla \om}{L^2}^{{3}/{2}}\right) \notag
        \\
        \le & C\varepsilon (\Gr^{5/2}(1+\log \Gr)^{{1}/{2}}\Sob{\nabla \om}{L^2}^2+\Gr^{{5}/{2}}\Sob{\nabla \om}{L^2}+\Gr^{3/2}\Sob{\Delta \om}{L^2}\Sob{\nabla \om}{L^2})\label{est:I2:appdx}.
    \end{align}
Taking the time average, applying the Cauchy-Schwarz inequality and using (\ref{est:time:om}-\ref{est:time:om:Lap}), we have
    \begin{align*}
        |\lbn I_2\rbn| &\le  C\varepsilon \left(\Gr^{5/2}(1+\log \Gr)^{{1}/{2}}\lbn \Sob{\nabla \om}{L^2}^2\rbn+\Gr^{{5}/{2}}\lbn \Sob{\nabla \om}{L^2}^2 \rbn^{1/2}+\Gr^{3/2} \lbn \Sob{\Delta \om}{L^2}^2 \rbn^{1/2} \lbn \Sob{\nabla \om}{L^2}^2 \rbn ^{1/2}\right)
        \\
        &\le C\varepsilon \Gr^{9/2}(1+\log \Gr)^{1/2}.
    \end{align*}
From \eqref{eq:NSE:rotatingvf}, we have
    \begin{align*}
        \partial_s \partial_y \overline{\om}=-\partial_y\overline{B}({\om},{\om})+\partial_y\Delta \overline{\om}+\partial_y \overline{f}.
    \end{align*}
Thus, we have
    \[
        I_3=-\frac{i\varepsilon}{2}(\widetilde{\om}^2 ,\partial_y\overline{B}(\om,\om))+\frac{i\varepsilon}{2}(\widetilde{\om}^2, \partial_y \overline{f})+\frac{i\varepsilon}{2}(\widetilde{\om}^2, \partial_y\Delta\overline{\om}).
    \]
To estimate $I_3$, we integrate by parts in each term and then apply Holder inequality, Agmon inequality, and Br\'ezis-Gallouet inequality to obtain
    \begin{align}
        |I_3|\le & C\varepsilon \Sob{\overline{B}(\om,\om)}{L^2_y}\int \Sob{\partial_y \widetilde{\om}}{L^2_y}\Sob{\widetilde{\om}}{L^\infty_y}\, dx+C\varepsilon \Sob{\overline{f}}{L^2_y} \int \Sob{\partial_y\widetilde{\om}}{L^2_y}\Sob{\widetilde{\om}}{L^\infty_y}\,dx \notag
        \\
        &+C\varepsilon \Sob{\partial_{yy} \overline{\om}}{L^2_y} \int \Sob{\partial_y \widetilde{\om}}{L^2_y}\Sob{\widetilde{\om}}{L^\infty_y}\,dx \notag
        \\
        \le & C\varepsilon \Sob{\nabla \om}{L^2}\Sob{u}{L^\infty}\int \Sob{\nabla \om}{L^2_y}^{3/2}\Sob{\om}{L^2_y}^{1/2}\, dx+C\varepsilon \Sob{ \overline{f}}{L^2} \int \Sob{\nabla \om}{L^2_y}^{3/2}\Sob{\om}{L^2_y}^{1/2}\,dx \notag
        \\
        &+C\varepsilon \Sob{\partial_{yy} \overline{\om}}{L^2} \int \Sob{\nabla \om}{L^2_y}^{3/2}\Sob{\om}{L^2_y}^{1/2}\,dx \notag
        \\
        \le & C\varepsilon\left(\Sob{\om}{L^2}^{{3}/{2}}\log\left(\frac{\Sob{\Delta \om}{L^2}}{c\Sob{\nabla \om}{L^2}}+1\right)^{1/2}\Sob{ \nabla \om}{L^2}^{5/2}+\Sob{f}{L^2}\Sob{\om}{L^2}^{{1}/{2}}\Sob{\nabla \om}{L^2}^{{3}/{2}}\right.  \notag
        \\ 
        &\left.\hspace{15 em}+\Sob{\Delta\om}{L^2}\Sob{\om}{L^2}^{{1}/{2}}\Sob{\nabla \om}{L^2}^{{3}/{2}}\right) \notag
        \\
        \le & C\varepsilon (\Gr^{5/2}(1+\log \Gr)^{{1}/{2}}\Sob{\nabla \om}{L^2}^2+\Gr^{{5}/{2}}\Sob{\nabla \om}{L^2}+\Gr^{3/2}\Sob{\Delta \om}{L^2}\Sob{\nabla \om}{L^2})\label{est:I3:appdx}.
    \end{align}
Taking the time average, we have
    \begin{align*}
        |\lbn I_3\rbn| \le C\varepsilon \Gr^{9/2}(1+\log \Gr)^{{1}/{2}}.
    \end{align*}
From \eqref{eq:NSE:rotatingvf}, we have
    \begin{align*}
        \partial_s \widetilde{\om}=-\widetilde{B}({\om},{\om})+\Delta \widetilde{\om}+ \widetilde{f}.
    \end{align*}
Thus
    \[
        J_2=i\veps(\check{f}, \widetilde{B}(\om,\om))-i\veps(\check{f}, \widetilde{f})-i\veps(\Check{f}, \Delta \widetilde{\om}).
    \]
Note that
    \begin{align}\notag
        (\check{f}, \widetilde{f})=\sum_{k_1\neq0}\frac{|k|^2}{k_1}|f_k|^2=\sum_{k_1>0}\frac{|k|^2}{k_1}|f_{(k_1,k_2)}|^2-\sum_{k_1>0}\frac{|k|^2}{k_1}|f_{(-k_1,k_2)}|^2.
    \end{align}
Since $f$ is real valued and satisfies symmetry conditions \eqref{eq:symmetry}, we have 
\[
        f^*_k=f_{-k},\quad f_{(k_1,-k_2)}=-f_{(k_1,k_2)}
\]
and
\[
        |f_{(-k_1,k_2)}|^2=|f_{(k_1,-k_2)}|^2=|f_{(k_1,k_2)}|^2
\]
Thus, we obtain
\[ (\check{f}, \widetilde{f})=0.\]
Applying the Cauchy-Schwarz inequality and Br\'ezis-Gallouet inequality, we obtain
    \begin{align}
        |J_2|&\le C\varepsilon \left(\Sob{\om}{L^2}\log\left(\frac{\Sob{\Delta \om}{L^2}}{c\Sob{\nabla \om}{L^2}}+1\right)^{1/2}\Sob{\nabla\om}{L^2}+\Sob{ f}{H^2}\Sob{\Delta \om}{L^2} \right)
        \notag\\
        &\le C\varepsilon\Gr(1+\log \Gr)^{{1}/{2}}\Sob{\nabla \om}{L^2}+C\Gr\Sob{\Delta \om}{L^2}\label{est:J2:appdx}.
    \end{align}
Taking the time average and using \eqref{est:time:om}, we have
    \begin{align*}
        |\lbn J_2\rbn| \le C\varepsilon \Gr^{3}(1+\log \Gr)^{{1}/{2}}.
    \end{align*}
Finally, we take the time average in \eqref{eq:NSE:rotatingv:non-zonal} and use the estimates for $I_2, I_3$ and $J_2$ to obtain
    \begin{align*}
        \lbn \Sob{\nabla\widetilde{\om}}{L^2}^2\rbn &\le -\lbn (B(\om, \om), \widetilde{\om} \rbn+ \lbn (f, \widetilde{\om})\rbn
        \\
        &\le |\lbn I_2\rbn|+|\lbn I_3\rbn|+|\lbn J_2\rbn|
        \\
        &\le C\varepsilon \Gr^{9/2}(1+\log \Gr)^{{1}/{2}}.
    \end{align*}
as claimed.
Next, we obtain the estimate for $\Sob{\widetilde{\omega}}{L^2}$. By Poincare inequality, we have
\[\Sob{\nabla \widetilde{\om}}{L^2}^2\ge \Sob{\widetilde{\om}}{L^2}^2.\]
Using this in \eqref{eq:NSE:rotatingv:non-zonal}, invoking the bounds in \eqref{est:I2:appdx}-\eqref{est:J2:appdx} and then integrating in time, we obtain
\begin{align*}
    &\Sob{\widetilde{\om}(t)}{L^2}^2+\int_0^t \Sob{\nabla \widetilde{\om}}{L^2}^2e^{s-t}\,ds\\&\le  e^{-t}\Sob{\widetilde{\om}(0)}{L^2}^2+C\varepsilon \Gr^{5/2}(1+\log \Gr)^{{1}/{2}}\int_0^t \Sob{\nabla \om(s)}{L^2}^2e^{s-t}\,ds\\
    &+C\veps \Gr^{{5}/{2}}\int_0^t\Sob{\nabla \om}{L^2}e^{s-t}\,ds+C\veps\Gr^{3/2}\int_0^t \Sob{\Delta \om}{L^2}\Sob{\nabla \om}{L^2}e^{s-t}\,ds\\
    &+C\veps\Gr \int_0^t \Sob{\Delta \om}{L^2}e^{s-t}\,ds.
\end{align*}
Applying the Cauchy–Schwarz inequality, we obtain
\begin{align*}
    &\Sob{\widetilde{\om}(t)}{L^2}^2+\int_0^t \Sob{\nabla \widetilde{\om}}{L^2}^2e^{s-t}\,ds\\
    &\le e^{-t}\Sob{\widetilde{\om}(0)}{L^2}^2+C\varepsilon \Gr^{5/2}(1+\log \Gr)^{{1}/{2}}\int_0^t \Sob{\nabla \om(s)}{L^2}^2e^{s-t}\,ds\\
    &+C\veps\Gr^{{5}/{2}}\left (\int_0^t\Sob{\nabla \om}{L^2}^2 e^{s-t}\,ds\right)^{1/2}+C\veps\Gr^{3/2}\left (\int_0^t \Sob{\Delta \om}{L^2}^2 e^{s-t}\right)^{1/2}\left(\int_0^t\Sob{\nabla \om}{L^2}^2e^{s-t}\,ds\right)^{1/2}\\
    &+C\veps\Gr \left (\int_0^t \Sob{\Delta \om}{L^2}^2 e^{s-t}\,ds\right)^{1/2}.
\end{align*}
We now shift the origin $t=0$ to a time $T_0$ sufficiently large as determined by \cref{thm:djoko:bounds} and use the estimates therein to obtain
\begin{align*}
    \sup_{s\ge T_0}\Sob{\widetilde{\om}(s)}{L^2}^2&\le C\veps \Gr^{9/2}(1+\log \Gr)^{1/2}+C\veps \Gr^{7/2}+C\veps \Gr^{9/2}+C\veps \Gr^{3}\\
    &\le C\veps \Gr^{9/2}(1+\log\Gr)^{1/2}.
\end{align*}
\end{proof}

\section{Proof of $\lb \bdy_t\Tr[\PP_\tN(t)\TT(t)\PP_\tN(t)]\rb=0$}\label{sect:app:B}

\begin{Lem}\label{lem:time:avg:bdd:op}
Let $\TT(t)$ be a bounded linear operator over $L^2$ such that
    \begin{align}\notag
        \|\TT\|:=\sup_{t\geq0}\|\TT(t)\|<\infty,
    \end{align}
where $\|\TT(t)\|$ denotes the corresponding operator norm. Then
    \begin{align}\notag
        \lb \bdy_t\Tr[\PP_\tN(t)\TT(t)\PP_\tN(t)]\rb=0.
    \end{align}
\end{Lem}

\begin{proof}
Recall that $\{\phi_j(t)\}_{j=1}^\tN$ is an orthonormal basis of $\PP_\tN L^2$ for each $t\geq 0$. Due to the invariance of the trace with respect to orthonormal basis, it follows that
    \begin{align}\notag
        \Tr[\PP_\tN(t)\TT(t)\PP_\tN(t)]&=\sum_{j=1}^\infty(\PP_\tN(t)\TT(t)\PP_\tN(t)\chi_j,\chi_j)=\sum_{j=1}^\infty(\TT(t)\PP_N(t)\chi_j,\PP_\tN(t)\chi_j)\notag
        \\
        &=\sum_{j=1}^\tN(\TT\PP_N\phi_j(t),\PP_\tN\phi_j(t))=\sum_{j=1}^\tN(\TT(t)\phi_j(t),\phi_j(t)).\notag
    \end{align}
By the Cauchy-Schwarz inequality, it follows that
    \begin{align}\notag
        \sup_{t\geq0}|\Tr[\PP_\tN(t)\TT(t)\PP_\tN(t)]|\leq\sup_{t\geq0}\sum_{j=1}^N\|\TT(t)\phi_j(t)\|_{L^2}\|\phi_j(t)\|_{L^2}\leq N\|\TT\|,
    \end{align}
where we have used the fact that the $\phi_j$ are orthonormal. Therefore
    \begin{align}
        \lb \bdy_t\Tr[\PP_\tN(t)\TT(t)\PP_\tN(t)]\rb&=\limsup_{t\goesto\infty}\frac{\Tr[\PP_\tN(t)\TT(t)\PP_\tN(t)-\Tr[\PP_\tN(0)\TT(0)\PP_\tN(0)]}{t}\notag
        \\
        &\leq\limsup_{t\goesto\infty}\frac{2N\|\TT\|}{t}=0,\notag
    \end{align}
as claimed.
\end{proof}

\hfill

\hfill

\noindent Aseel Farhat\\
{\footnotesize
Department of Mathematics\\
Florida State University \\
Email: \url{afarhat@fsu.edu} \\
Department of Mathematics \\
University of Virginia \\
Email: \url{af7py@virginia.edu}\\
}

\noindent Anuj Kumar\\
{\footnotesize
Department of Mathematics\\
Florida State University \\
Email: \url{akumar241@outlook.com}\\
}

\noindent Vincent R. Martinez\\
{\footnotesize
Department of Mathematics \& Statistics\\
CUNY Hunter College \\
Department of Mathematics \\
CUNY Graduate Center \\
Web: \url{http://math.hunter.cuny.edu/vmartine/}\\
Email: \url{vrmartinez@hunter.cuny.edu}
}

\end{document}